\documentclass[]{IEEEtran} 
\usepackage{times}  
\usepackage{helvet} 
\usepackage{courier}  
\usepackage[hyphens]{url}  
\usepackage{epsfig}
\usepackage{amsmath}
\usepackage{amssymb}
\usepackage[left=0.75in,top=1.0in,right=0.75in,bottom=1.0in]{geometry}
\usepackage{cite}
\usepackage{bm}
\usepackage{graphicx}           
\usepackage{psfrag}             
\usepackage{subfigure}          
\usepackage{empheq}
\usepackage{enumerate}
\usepackage{multirow}
\usepackage{booktabs}
\usepackage{natbib}

\usepackage{hyperref}       
\usepackage{url}            
\usepackage{booktabs}       
\usepackage{amsfonts}       
\usepackage{nicefrac}       
\usepackage{microtype}      
\usepackage{breakcites}
\usepackage{microtype}

\usepackage{amsthm}

\usepackage{tikz}
\usetikzlibrary{arrows}

\newtheoremstyle{exampstyle}
  {\topsep 3pt} 
  {\topsep 3pt} 
  {} 
  {} 
  {\bfseries} 
  {.} 
  {.5em} 
  {} 

\theoremstyle{exampstyle} \newtheorem{assumption}{Assumption}
\newtheorem{theorem}{Theorem}
\newtheorem{lemma}{Lemma}

\newtheorem{proposition}{Proposition}

\usepackage{algorithm}
\usepackage{algorithmic}
\usepackage{setspace}
\let\Algorithm\algorithm
\renewcommand\algorithm[1][]{\Algorithm[#1]\setstretch{2}}

\newcommand{\beginsupplement}{%
        \setcounter{lemma}{0}
        \renewcommand{\thelemma}{S\arabic{lemma}}%
        \setcounter{equation}{0}
        \renewcommand{\theequation}{S\arabic{equation}}%
        \setcounter{section}{0}
        \renewcommand{\thesection}{S\arabic{section}}%
     }

\title{{\LARGE{\bf Distributed Deep Learning with Event-Triggered Communication}}}

\author{Jemin George\textsuperscript{\rm 1}, Prudhvi Gurram\textsuperscript{\rm 1, \rm 2}\\
\textsuperscript{\rm 1}CCDC Army Research Laboratory\\ 
Adelphi, MD 20783\\
jemin.george.civ@mail.mil \\ 
\textsuperscript{\rm 2}Booz Allen Hamilton\\ 
McLean, VA 22102\\
gurram\_prudhvi@bah.com 
}

\begin{document}

\maketitle

\begin{abstract}
 We develop a Distributed Event-Triggered Stochastic GRAdient Descent (DETSGRAD) algorithm for solving non-convex optimization problems typically encountered in distributed deep learning. We propose a novel communication triggering mechanism that would allow the networked agents to update their model parameters aperiodically and provide sufficient conditions on the algorithm step-sizes that guarantee the asymptotic mean-square convergence. The algorithm is applied to a distributed supervised-learning problem, in which a set of networked agents collaboratively train their individual neural networks to recognize handwritten digits in images, while aperiodically sharing the model parameters with their one-hop neighbors. Results indicate that all agents report similar performance that is also comparable to the performance of a centrally trained neural network, while the event-triggered communication provides significant reduction in inter-agent communication. Results also show that the proposed algorithm allows the individual agents to recognize the digits even though the training data corresponding to all the digits are not locally available to each agent.
\end{abstract}

\section{Introduction}\label{sec:Intro}

With the advent of smart devices, there has been an exponential growth in the amount of data collected and stored locally on individual devices. Applying machine learning to extract value from such massive data to provide data-driven insights, decisions, and predictions has been a popular research topic as well as the focus of numerous businesses. However, porting these vast amounts of data to a data center to conduct traditional machine learning has raised two main issues: (i) the communication challenge associated with transferring vast amounts of data from a large number of devices to a central location and (ii) the privacy issues associated with sharing raw data. Distributed machine learning techniques based on the server-client architecture~\citep{Mu2014, NIPS2014_5597, ICCCN2017} have been proposed as solutions to this problem. On one extreme end of this architecture, we have the \emph{parameter server} approach, where a server or group of servers initiate distributed learning by pushing the current model to a set of client nodes that host the data. The client nodes compute the local gradients or parameter updates and communicate them to the server nodes. Server nodes aggregate these values and update the current model~\citep{Zhang2018, NIPS2014_5597}. On the other extreme, we have \emph{federated learning}, where each client node obtains a local solution to the learning problem and the server node computes a global model by averaging the local models~\citep{Jakub2016, McMahan2017}. Besides the server-client architecture, a \emph{shared-memory} (multicore/multiGPU) architecture, where different processors independently compute the gradients and update the global model parameter using a shared memory, has also been proposed as a solution to the distributed machine learning problem~\citep{NIPS2011_4390, NIPS2015_5717, NIPS2015_6031, Feyzmahdavian2016}. However, none of the above-mentioned learning techniques are truly distributed since they follow a master-slave architecture and do not involve any peer-to-peer communication. Furthermore, these techniques are not always robust and they are rendered useless if the master/server node or the shared-memory fails. Therefore, we aim to develop a fully distributed machine learning architecture enabled by client-to-client interaction.

For large-scale machine learning, stochastic gradient descent (SGD) methods are often preferred over batch gradient methods~\citep{Bottou2018SIAM} because (i) in many large-scale problems, there is a good deal of redundancy in data and therefore it is inefficient to use all the data in every optimization iteration, (ii) the computational cost involved in computing the batch gradient is much higher than that of the stochastic gradient, and (iii) stochastic methods are more suitable for online learning where data are arriving sequentially. Since most machine learning problems are non-convex, there is a need for distributed stochastic gradient methods for non-convex problems. Therefore, here we present a communication efficient, distributed stochastic gradient algorithm for non-convex problems and demonstrate its utility for distributed machine learning.

\section{Related work}

\subsection{Distributed Non-Convex Optimization} 

A few early examples of (non-stochastic or deterministic) distributed non-convex optimization algorithms include the Distributed Approximate Dual Subgradient (DADS) Algorithm \citep{Zhu2013}, NonconvEx primal-dual SpliTTing (NESTT) algorithm \citep{NESTT2016}, and the Proximal Primal-Dual Algorithm (Prox-PDA) \citep{pmlr-v70-hong17a}. More recently, a non-convex version of the accelerated distributed augmented Lagrangians (ADAL) algorithm is presented in \citet{Chatzipanagiotis2017} and successive convex approximation (SCA)-based algorithms such as iNner cOnVex Approximation (NOVA) and in-Network succEssive conveX approximaTion algorithm (NEXT) are given in \citet{Scutari2017} and \citet{Lorenzo2016}, respectively. References \citep{Hong2018, Guo2017, Hong2016SIAM} provide several distributed alternating direction method of multipliers (ADMM) based non-convex optimization algorithms. Non-convex versions of Decentralized Gradient Descent (DGD) and Proximal Decentralized Gradient Descent (Prox-DGD) are given in \citet{Zeng2018TSP}. Finally, Zeroth-Order NonconvEx (ZONE) optimization algorithms for mesh network (ZONE-M) and star network (ZONE-S) are presented in \citet{Hajinezhad2019}. However, almost all aforementioned \emph{consensus optimization} algorithms focus on non-stochastic problems and are extremely communication heavy because they require constant communication among the agents.

\subsection{Distributed Convex SGD} 

Within the consensus optimization literature, there exist several works on distributed stochastic gradient methods, but mainly for strongly convex optimization problems. These include the stochastic subgradient-push method for distributed optimization over time-varying directed graphs  given in \citet{Nedic2016}, distributed stochastic optimization over random networks given in \citet{2018arXiv180307836J}, the Stochastic Unbiased Curvature-aided Gradient (SUCAG) method given in \citet{SUCAG}, and distributed stochastic gradient tracking methods \citet{Pu2018ArxiV}. There are very few works on distributed stochastic gradient methods for non-convex optimization~\citep{Tatarenko2017,Bianchi2013}; however, the push-sum algorithm given in \citet{Tatarenko2017} assumes there are no saddle-points and it often requires up to 3 times as many internal variables as the proposed algorithm. Compared to \citet{Bianchi2013} and \citet{Tatarenko2017}, the proposed algorithm provides an explicit consensus rate and allows the parallel execution of the consensus communication and gradient computation steps.

\subsection{Parallel SGD} 

There exist numerous asynchronous SGD algorithms aimed at parallelizing the data-intensive machine learning tasks. The two popular asynchronous parallel implementations of SGD are the computer network implementation originally proposed in \citet{NIPS2011_4247} and the shared memory implementation introduced in \citet{NIPS2011_4390}. Computer network implementation follows the master-slave architecture and \citet{NIPS2011_4247} showed that for smooth convex problems, the delays due to asynchrony are asymptotically negligible. \citet{Feyzmahdavian2016} extend the results in \citet{NIPS2011_4247} for regularized SGD.  Extensions of the computer network implementation of asynchronous SGD with variance reduction and polynomially growing delays are given in \citet{2016arXiv160403584H} and \citet{pmlr-v80-zhou18b}, respectively. \citet{NIPS2011_4390} proposed a lock-free asynchronous parallel implementation of SGD on a shared memory system and proved a sublinear convergence rate for strongly convex smooth objectives. The lock-free algorithm, HOGWILD!, proposed in \citet{NIPS2011_4390} has been applied to PageRank approximation \citep{Mitliagkas2015}, deep learning~\citep{noel2014dogwild}, and recommender systems~\citep{6413853}. In \citet{NIPS2013_4939}, authors extended the HOGWILD! algorithm to a dual averaging algorithm that works for non-smooth, non-strongly convex problems with sparse gradients. An extension of HOGWILD! called BUCKWILD! is introduced in \citet{NIPS2015_5717} to account for quantization errors introduced by fixed-point arithmetic. In \citet{NIPS2015_6031}, the authors show that because of the noise inherent to the sampling process within SGD, the errors introduced by asynchrony in the shared-memory implementation are asymptotically negligible. A detailed comparison of both computer network and shared memory implementation is given in \citet{2015arXiv150608272L}. Again, the aforementioned asynchronous algorithms are not distributed since they rely on a shared-memory or central coordinator.

\subsection{Decentralized SGD} 

Recently, numerous \emph{decentralized SGD} algorithms for non-convex optimization have been proposed as a solution to the communication bottleneck often encountered in the server-client architecture~\citep{LianNIPS2017,JiangNIPS2017,Tang2018ICML,Lian2018ICML,Wang2018arXiv,Haddadpour2019ICML,Mahmoud2019ICML,Wang2019arXiv}. However almost all these works primarily focus on the performance of the algorithm during a fixed time interval, and the constant algorithm step-size, which often depends on the final time, is selected to speed-up the convergence rate. These SGD algorithms with constant step-size can only guarantee convergence to some $\epsilon$-ball of the stationary point. Furthermore, most of the aforementioned decentralized SGD algorithms provide convergence rates in terms of the average of all local estimates of the global minimizer without ever proving a similar or faster consensus rate. In fact, most decentralized SGD algorithms can only provide bounded consensus and they require a centralized averaging step after running the algorithm until the final-time~\cite{LianNIPS2017,Tang2018ICML,Lian2018ICML,Haddadpour2019ICML,Wang2019arXiv}. Finally, most application of decentralized SGD focus on distributed learning scenarios where the data is distributed identically across all agents.

\subsection{Contribution} 

Currently, there exists no distributed SGD algorithm for the non-convex problems that doesn't require constant or periodic communication among the agents. In fact, algorithms in \citep{LianNIPS2017,JiangNIPS2017,Tang2018ICML,Lian2018ICML,Wang2018arXiv,Haddadpour2019ICML,Mahmoud2019ICML,Wang2019arXiv} all rely on periodic communication despite the local model has not changed from previously communicated model. This is a waste of resources, especially in wireless setting and therefore we propose an approach that would allow the nodes to transmit only if the local model has significantly changed from previously communicated model. The contributions of this paper are three-fold: (i) we propose a fully distributed machine learning architecture, (ii) we present a distributed SGD algorithm built on a novel communication triggering mechanism, and provide sufficient conditions on step-sizes such that the algorithm is mean-square convergent, and (iii) we demonstrate the efficacy of the proposed event-triggered SGD algorithm for distributed supervised learning with i.i.d. and more importantly, non-i.i.d. data.


\subsection*{Notation}

Let $\mathbb{R}^{n\times m}$ denote the set of $n\times m$ real matrices. For a vector $\bm{\phi}$, $\phi_i$ is the $i-{\text{th}}$ entry of $\bm{\phi}$.  An $n\times n$ identity matrix is denoted as $I_n$ and $\mathbf{1}_n$ denotes an $n$-dimensional vector of all ones. For $p\in[1,\,\infty]$, the $p$-norm of a vector $\mathbf{x}$ is denoted as $\left\| \mathbf{x} \right\|_p$. For matrices $A \in \mathbb{R}^{m\times n}$ and $B \in \mathbb{R}^{p \times q}$, $A \otimes B \in \mathbb{R}^{mp \times nq}$ denotes their Kronecker product.

For a graph $\mathcal{G}\left(\mathcal{V},\mathcal{E}\right)$ of order $n$, $\mathcal{V} \triangleq \left\{1, \ldots, n\right\}$ represents the agents or nodes and the communication links between the agents are represented as $\mathcal{E} \triangleq \left\{e_1, \ldots, e_{\ell}\right\} \subseteq \mathcal{V} \times \mathcal{V}$. Let $\mathcal{N}_i \triangleq \left\{j \in \mathcal{V}~:~(i,j)\in\mathcal{E}\right\}$ denote the set of neighbors of node $i$. Let $\mathcal{A} = \left[a_{ij}\right]\in \mathbb{R}^{n\times n}$ be the \emph{adjacency matrix} with entries of $a_{ij} = 1 $ if $(i,j)\in\mathcal{E}$ and zero otherwise. Define $\Delta = \text{diag}\left(\mathcal{A}\mathbf{1}_n\right)$ as the in-degree matrix and $\mathcal{L} = \Delta - \mathcal{A}$ as the graph \emph{Laplacian}. 

\section{Distributed machine learning}\label{sec:Problem}

Our problem formulation closely follows the centralized machine learning problem discussed in \citet{Bottou2018SIAM}. Consider a networked set of $n$ agents, each with a set of $m_i$, $i=1,\ldots,n$, independently drawn input-output samples $\{\bm{x}_i^{j},\,\bm{y}_i^j\}_{j=1}^{j=m_i}$, where $\bm{x}_i^j\in\mathbb{R}^{d_x}$ and $\bm{y}_i^j\in\mathbb{R}^{d_y}$ are the $j$-th input and output data, respectively, associated with the $i$-th agent. For example, the input data could be images and the outputs could be labels. Let $h\left(\cdot\,;\,\cdot\right):\mathbb{R}^{d_x}\times\mathbb{R}^{d_w} \mapsto \mathbb{R}^{d_y}$ denote the prediction function, fully parameterized by the vector $\bm{w}\in\mathbb{R}^{d_w}$. Each agent aims to find the parameter vector that minimizes the losses, $\ell\left(\cdot\,;\,\cdot\right):\mathbb{R}^{d_y}\times\mathbb{R}^{d_y} \mapsto \mathbb{R}$, incurred from inaccurate predictions. Thus, the loss function $\ell\left(h\left(\bm{x}_i;\bm{w}\right),\bm{y}_i\right)$ yields the loss incurred by the $i$-th agent, where $h\left(\bm{x}_i;\bm{w}\right)$ and $\bm{y}_i$ are the predicted and true outputs, respectively, for the $i$-th node.


Assuming the input output space $\mathbb{R}^{d_x}\times\mathbb{R}^{d_y}$ associated with the $i$-th agent is endowed with a probability measure $P_i~:~\mathbb{R}^{d_x}\times\mathbb{R}^{d_y}\mapsto[0,\,1]$, the objective function an agent wishes to minimize is
\begin{align}\label{Eq:Risk}
\begin{split}
R_i(\bm{w}) &= \int_{\mathbb{R}^{d_x}\times\mathbb{R}^{d_y}}\, \ell\left(h\left(\bm{x}_i;\bm{w}\right),\bm{y}_i\right)\,dP_i\left(\bm{x}_i,\bm{y}_i\right) \\
&= \mathbb{E}_{P_i} \left[ \ell\left(h\left(\bm{x}_i;\bm{w}\right),\bm{y}_i\right) \right].
\end{split}
\end{align}
Here $R_i(\bm{w})$ denotes the expected risk given a parameter vector $\bm{w}$ with respect to the probability distribution $P_i$. The total expected risk across all networked agents is given as
\begin{equation}\label{Eq:TotalRisk}
R(\bm{w}) = \sum_{i=1}^{n}R_i(\bm{w}) = \sum_{i=1}^{n} \mathbb{E}_{P_i} \left[ \ell\left(h\left(\bm{x}_i;\bm{w}\right),\bm{y}_i\right) \right].
\end{equation}
Minimizing the expected risk is desirable but often unattainable since the distributions $P_i$ are unknown. Thus, in practice, each agent chooses to minimize the empirical risk $\bar{R}_i(\bm{w})$ defined as
\begin{equation}\label{Eq:Risk1}
\bar{R}_i(\bm{w}) = \frac{1}{m_i}\sum_{j=1}^{m_i} \ell\left(h\left(\bm{x}_i^j;\bm{w}\right),\bm{y}_i^j\right).
\end{equation}
Here, the assumption is that $m_i$ is large enough so that $\bar{R}_i(\bm{w}) \approx R_i(\bm{w})$. The total empirical risk across all networked agents is
\begin{equation}\label{Eq:TotalRisk1}
\bar{R}(\bm{w}) = \sum_{i=1}^{n}\bar{R}_i(\bm{w}) = \sum_{i=1}^{n}  \frac{1}{m_i}\sum_{j=1}^{m_i} \ell\left(h(\bm{x}_i^j;\bm{w}),\bm{y}_i^j\right)
\end{equation}

To simplify the notation, let us represent a sample input-output pair $(\bm{x}_i,\,\bm{y}_i)$ by a random seed $\bm{\xi}_i$ and let $\bm{\xi}_i^j$ denote the $j$-th sample associated with the $i$-th agent. Define the loss incurred for a given $\left(\bm{w},\bm{\xi}_i^j\right)$ as $\ell\left(\bm{w},\bm{\xi}_i^j\right)$. Now, the distributed learning problem can be posed as an optimization involving sum of local empirical risks, i.e.,
\begin{align}\label{eq:Opt1}
  \min_{\bm{w}}\, f(\bm{w}) = \min_{\bm{w}}\,\sum_{i=1}^{n} \, f_i\left( \bm{w} \right),
\end{align}
where $f_i\left( \bm{w} \right) = \frac{1}{m_i}\sum_{j=1}^{m_i} \ell\left(\bm{w},\bm{\xi}_i^{j}\right)$.

\section{Distributed event-triggered SGD}\label{sec:SGD}

Here we propose a distributed event-triggered stochastic gradient method to solve \eqref{eq:Opt1}. Let $\bm{w}_i(k) \in \mathbb{R}^{d_w}$ denote agent $i$'s estimate of the optimizer at time instant $k$. Thus, for an arbitrary initial condition $\bm{w}_i(0)$, the update rule at node $i$ is as follows:
\begin{align}\label{Eq:DSG}
\begin{split}
  \bm{w}_i(k + 1) = \bm{w}_i(k) &- \beta_k \, \sum_{j=1}^{n}\,a_{ij} \left( \hat{\bm{w}}_i(k) - \hat{\bm{w}}_j(k) \right) \\ &- \alpha_k \,\mathbf{ g}_i\left( \bm{w}_i(k), \bm{\xi}_i(k) \right),
\end{split}
\end{align}
where $\alpha_k$ and $\beta_k$ are hyper parameters to be specified, $a_{ij}$ are the entries of the adjacency matrix and $\mathbf{g}_i\left( \bm{w}_i(k), \bm{\xi}_i(k) \right)$ represents either a simple stochastic gradient, mini-batch stochastic gradient or a stochastic quasi-Newton direction, i.e.,
\begin{align*}
  \mathbf{g}_i\left( \bm{w}_i(k), \bm{\xi}_i(k) \right) = \left\{
                                              \begin{array}{l}
                                                \nabla \ell\left(\bm{w}_i(k),\bm{\xi}_i^{k}\right), ~~\mbox{or} \\
                                                \frac{1}{n_i(k)}\sum\limits_{s=1}^{n_i(k)}\,\nabla \ell\left(\bm{w}_i(k),\bm{\xi}_i^{k,s}\right), ~~\mbox{or} \\
                                                H_i(k) \frac{1}{n_i(k)}\sum\limits_{s=1}^{n_i(k)}\,\nabla \ell\left(\bm{w}_i(k),\bm{\xi}_i^{k,s}\right),
                                              \end{array}
                                            \right.
\end{align*}
\noindent where $n_i(k)$ denotes the mini-batch size, $H_i(k)$ is a positive definite scaling matrix, $\bm{\xi}_i^{k}$ represents the single random input-output pair sampled at time instant $k$, and $(\bm{\xi}_i^{k,s})$ denotes the $s$-th input-output pair out of the $n_i(k)$ random input-output pairs sampled at time instant $k$. For $i=1,\ldots,n$, the piece-wise constant signal $\hat{\bm{w}}_i(k)$ defined as
\begin{equation}
    \hat{\bm{w}}_i(k) = \bm{w}_i(t_q^i), \,\,\forall \,k\in \left\{ t_q^i,\,t_q^i+1,\,\ldots,\,t_{q+1}^i-1 \right\},
\end{equation}
denote agent $i$'s last broadcasted estimate of the optimizer. Here $\left\{t_q^i,\,\,\,q=0,1,\ldots\right\}$ with $t_0^i=0$ denotes triggering instants, i.e., the time instants when agent $i$ broadcasts $\bm{w}_i$ to its neighbors. Define $\mathbf{w}(k) \triangleq \begin{bmatrix} \bm{w}_1^\top(k) & \ldots & \bm{w}_n^\top(k) \end{bmatrix}^\top \in \mathbb{R}^{nd_w}$ and
$\hat{\mathbf{w}}(k) \triangleq \begin{bmatrix} \hat{\bm{w}}_1^\top(k) & \ldots & \hat{\bm{w}}_n^\top(k) \end{bmatrix}^\top \in \mathbb{R}^{nd_w}$. Now \eqref{Eq:DSG} can be written as
\begin{align}\label{Eq:DSG1}
\begin{split}
  \mathbf{w}(k + 1) = \mathbf{w}(k) &-\beta_k \left(\mathcal{L} \otimes I_{d_w}\right) \hat{\mathbf{w}}(k) \\&- \alpha_k \, \mathbf{g}(\mathbf{w}(k),\bm{\xi}(k)),
  \end{split}
\end{align}
where $\mathcal{L}$ is the network Laplacian and
$$\mathbf{g}(\mathbf{w}(k),\bm{\xi}(k)) \triangleq \begin{bmatrix}
                             \mathbf{g}_1\left( \bm{w}_1(k), \bm{\xi}_1(k) \right) \\
                             \vdots \\
                             \mathbf{g}_n\left( \bm{w}_n(k), \bm{\xi}_n(k) \right)
                           \end{bmatrix} \in \mathbb{R}^{n d_w}.$$
Let $\bm{e}_i(k) = {\bm{w}}_i(k)-\hat{\bm{w}}_i(k)$ and $\mathbf{e}(k) = \mathbf{w}(k) - \hat{\mathbf{w}}(k)$. Now \eqref{Eq:DSG1} can be written as
\begin{align}\label{Eq:DSG1a}
\begin{split}
  \mathbf{w}(k + 1) = \left(\mathcal{W}_k \otimes I_{d_w}\right) \mathbf{w}(k) &+\beta_k\left(\mathcal{L} \otimes I_{d_w}\right) {\mathbf{e}}(k) \\
	&- \alpha_k \, \mathbf{g}(\mathbf{w}(k),\bm{\xi}(k)),
\end{split}
\end{align}
where $\mathcal{W}_k = \left(I_{n}-\beta_k\mathcal{L}\right)$. The event instants are defined as
\begin{equation}\label{TriggCond}
    t^i_{q+1} = \inf\left\{ k > t^i_q \,\, | \,\, \left\|\bm{e}_i(k)\right\|_1 \geq \upsilon_0\,\alpha_k \right\},
\end{equation}
where $\upsilon_0$ is a positive constant to be defined. Pseudo-code of the proposed distributed event-triggered SGD is given in Algorithm~\ref{Algorithm1} (see supplementary material).


Now we state the following assumption on the individual objective functions:
\begin{assumption}\label{Assump:Lipz}
  Objective functions $f_i(\,\cdot\,)$ and its gradients $\nabla f_i(\,\cdot\,)$ $:\mathbb{R}^{d_w}\mapsto \mathbb{R}^{d_w}$ are Lipschitz continuous with Lipschitz constants $L^0_i > 0$ and $L_i > 0$, respectively, i.e., $\forall \,\bm{w}_a,\,\bm{w}_b\in\mathbb{R}^{d_w}, \, i=1,\ldots,n$, we have
\begin{align*}
  \| f_i(\,\bm{w}_a\,) -  f_i(\,\bm{w}_b\,) \|_2 &\leq L_i^0 \|\bm{w}_a-\bm{w}_b\|_2\,\,\textnormal{and} \\
	\,\,\| \nabla f_i(\,\bm{w}_a\,) - \nabla f_i(\,\bm{w}_b\,) \|_2 &\leq L_i \|\bm{w}_a-\bm{w}_b\|_2.
\end{align*}
\end{assumption}
Now we introduce $F(\cdot):\mathbb{R}^{nd_w}\mapsto \mathbb{R}$, an aggregate objective function of local variables
\begin{align}\label{eq:obj}
   F(\mathbf{w}(k)) = \sum_{i=1}^{n} \, f_i\left( \bm{w}_i(k) \right).
\end{align}
Following Assumption \ref{Assump:Lipz}, the function $F(\cdot)$ is Lipschitz continuous with Lipschitz continuous gradient $\nabla F(\cdot)$, i.e., $\forall \,\mathbf{w}_a,\,\mathbf{w}_b\in\mathbb{R}^{n d_w}$, we have $\| \nabla F(\,\mathbf{w}_a\,) - \nabla F(\,\mathbf{w}_b\,) \|_2 \leq L \|\mathbf{w}_a-\mathbf{w}_b\|_2$,
with constant $L = \max\limits_i\{L_i\}$ and $\nabla F(\,\mathbf{w}\,) \triangleq \begin{bmatrix} \nabla f_1(\,\bm{w}_1\,)^\top &\ldots & \nabla f_n(\,\bm{w}_n\,)^\top \end{bmatrix}^\top \in \mathbb{R}^{n d_w}$.
\begin{lemma}\label{Lemma:Lipz}
  Given Assumption \ref{Assump:Lipz}, we have $\forall \,\mathbf{w}_a,\,\mathbf{w}_b\in\mathbb{R}^{n d_w}$,
\begin{align}\label{Eqn:lemma1}
\begin{split}
  F(\,\mathbf{w}_b\,)
\leq F(\,\mathbf{w}_a\,) &+ \nabla F\left(\mathbf{w}_a\right)^\top(\mathbf{w}_b-\mathbf{w}_a) \\
&+ \frac{1}{2} L  \|\mathbf{w}_b-\mathbf{w}_a\|_2^2.
\end{split}
\end{align}
\end{lemma}
\begin{proof}
  Proof follows from the mean value theorem.
\end{proof}
\begin{assumption}\label{Assump:Fmin}
  The function $F(\cdot)$ is lower bounded by $F_{\inf}$, i.e., $F_{\inf} \leq F(\mathbf{w}), \,\, \forall\,\mathbf{w}\in\mathbb{R}^{nd_w}$.
\end{assumption}
Without loss of generality, we assume that $F_{\inf} \geq 0$.
\noindent Now we make the following assumption regarding $\{\alpha_k\}$ and $\{\beta_k\}$:
\begin{assumption}\label{Assump:AlphaBeta}
Sequences $\{\alpha_k\}$ and $\{\beta_k\}$ are selected as
\begin{align}\label{Eqn:AlphaBeta}
\alpha_k = \frac{a}{(k+1)^{\delta_2}} \quad \textnormal{and} \quad \beta_k = \frac{b}{(k+1)^{\delta_1}},
\end{align}
where $a > 0$, $b>0$, $0 < 3\delta_1 < \delta_2 \leq 1$, $\delta_1/2 + \delta_2 > 1$, and $\delta_2 > 1/2$.
\end{assumption}
For sequences $\{\alpha_k\}$ and $\{\beta_k\}$ that satisfy Assumption \ref{Assump:AlphaBeta}, we have $\sum_{k=1}^{\infty}\,\alpha_k = \infty$, $\sum_{k=1}^{\infty}\,\beta_k = \infty$, $\sum_{k=1}^{\infty}\,\alpha_k^2 < \infty$ and $\sum_{k=1}^{\infty}\,\alpha_k\beta_k^{1/2} < \infty$. Thus $\alpha_k$ and $\beta_k$ are not summable sequences. However, $\alpha_k$ is square-summable and $\alpha_k\sqrt{\beta_k}$ is summable.
\begin{assumption}\label{Assump:Graph}
  The interaction topology of $n$ networked agents is given as a connected undirected graph $\mathcal{G}\left(\mathcal{V},\mathcal{E}\right)$.
\end{assumption}
\begin{lemma}\label{Lemma2}
  Given Assumption \ref{Assump:Graph}, for all $\mathbf{x}\in\mathbb{R}^n$ we have $\mathbf{x}^\top\mathcal{L}\mathbf{x} = \tilde{\mathbf{x}}^\top\mathcal{L}\tilde{\mathbf{x}}  \geq  \lambda_2(\mathcal{L})\|\tilde{\mathbf{x}}\|_2^2$,
where $\tilde{\mathbf{x}} = \left(I_n - \frac{1}{n}\mathbf{1}_n\mathbf{1}^\top_n\right) \mathbf{x}$ is the average-consensus error and $\lambda_2(\cdot)$ denotes the smallest non-zero eigenvalue.
\end{lemma}
\begin{proof}
  This Lemma follows from the Courant-Fischer Theorem~\citep{Horn}.
\end{proof}
\begin{assumption}\label{Assump:Beta}
Parameter $b$ in sequence $\{\beta_k\}$ is selected such that
\begin{align}
\mathcal{W}_0 = \left(I_{n}-b\mathcal{L}\right)
\end{align}
has a single eigenvalue at $1$ corresponding to the right eigenvector $\mathbf{1}_n$ and the remaining $n-1$ eigenvalues of $\mathcal{W}_0$ are strictly inside the unit circle.
\end{assumption}
In other words, $b$ is selected such that $b < 1/\sigma_{\max}(\mathcal{L})$, where $\sigma_{\max}(\cdot)$ denotes the largest singular value. Thus, $b\sigma_{\max}(\mathcal{L}) < 1$.
Let $\mathbb{E}_{\xi}[\cdot]$ denote the expected value taken with respect to the distribution of the random variable $\bm{\xi}_k$ given the filtration $\mathcal{F}_{k}$ generated by the sequence $\{\mathbf{w}_{0},\ldots,\mathbf{w}_{k}\}$, i.e.,
\begin{align*}
   \mathbb{E}_{\xi}[\,\mathbf{w}_{k+1}\,] &= \mathbb{E}[\,\mathbf{w}_{k+1}\,|\mathcal{F}_{k}] \\
	&= \left(\mathcal{W}_k \otimes I_{d_w} \right) \mathbf{w}_k - \alpha_k \mathbb{E}[\,\mathbf{g}(\mathbf{w}_k,\bm{\xi}_k)\,|\mathcal{F}_{k}]\,\,\textnormal{a.s.},
\end{align*}
where $\textnormal{a.s.}$ (almost surely) denote events that occur with probability one. Now we make the following assumptions regarding the stochastic gradient term $\mathbf{g}(\mathbf{w}(k),\bm{\xi}(k))$.
\begin{assumption}\label{Assump:Grad1}
  Stochastic gradients are unbiased such that
\begin{equation}
  \mathbb{E}_{\xi}\left[\,  \mathbf{g}(\mathbf{w}_k,\bm{\xi}_k) \,\right] = \nabla F(\mathbf{w}_k), \,\, \textnormal{a.s.}
\end{equation}
That is to say
$
  \mathbb{E}_{\xi}\left[\,  \mathbf{g}(\mathbf{w}_k,\bm{\xi}_k) \,\right] =
\begin{bmatrix}
\mathbb{E}_{\xi_1}\left[\,\mathbf{g}_1\left( \bm{w}_1(k), \bm{\xi}_1(k) \right) \,\right]\\
\vdots \\
\mathbb{E}_{\xi_n}\left[\,\mathbf{g}_n\left( \bm{w}_n(k), \bm{\xi}_n(k) \right) \,\right]
\end{bmatrix}
=
\begin{bmatrix}
  \nabla f_1(\,\bm{w}_1(k)\,) \\
  \vdots \\
  \nabla f_n(\,\bm{w}_n(k)\,)
\end{bmatrix}
$
\end{assumption}
\begin{assumption}\label{Assump:Grad2}
 Stochastic gradients have conditionally bounded second moment, i.e., there exist scalars $\bar{\mu}_{v_1} \geq 0$ and $\bar{\mu}_{v_2} \geq 0$ such that
\begin{align}
\begin{split}
\mathbb{E}_{\xi} \left[ \| \mathbf{g}(\mathbf{w}_k,\bm{\xi}_k) \|_2^2 \right] \leq
\bar{\mu}_{v_1} + \bar{\mu}_{v_2} \left\|  \nabla F(\mathbf{w}_k) \right\|^2_2,\,\, \textnormal{a.s.}
\end{split}\label{Eq:2ndGrad}
\end{align}
\end{assumption}
Assumption~\ref{Assump:Grad2} is the bounded variance assumption typically made in all SGD literature.



\section{Convergence analysis}
Define the average-consensus error as $\tilde{\mathbf{w}}_k = \left(M \otimes I_{d_w}\right)\mathbf{w}_k$, where $M = I_{n} - \frac{1}{n}\mathbf{1}_{n}\mathbf{1}_{n}^\top$. Note that $M\mathcal{L} = \mathcal{L}$ and $\left(\mathcal{L} \otimes I_{d_w}\right)\tilde{\mathbf{w}}_k = \left(\mathcal{L} \otimes I_{d_w}\right)\mathbf{w}_k$. Thus from \eqref{Eq:DSG1a} we have
\begin{align}\label{Eqn:wTilde}
\begin{split}
    \tilde{\mathbf{w}}_{k+1} = &\left(\mathcal{W}_k \otimes I_{d_w} \right) \tilde{\mathbf{w}}_k +\beta_k \left(\mathcal{L} \otimes I_{d_w}\right) {\mathbf{e}}(k)\\ &\qquad- \alpha_k \left(M \otimes I_{d_w}\right) \mathbf{g}(\mathbf{w}_k,\bm{\xi}_k).
\end{split}
\end{align}
Our strategy for proving the convergence of the proposed distributed event-triggered SGD algorithm to a critical point is as follows. First we show that the consensus error among the agents are diminishing at the rate of $O\left(\frac{1}{(k+1)^{\delta_2}}\right)$ (see Theorem~\ref{Theorem:Consensus}). Asymptotic convergence of the algorithm is then proved in Theorem~\ref{Theorem:Convergence}. Theorem~\ref{Theorem:SummableGrad} then establishes that the weighted expected average gradient norm is a summable sequence. Convergence rate of the algorithm in the typical weak sense is given in Theorem~\ref{Theorem:ConvergenceRate}. Finally, Theorem~\ref{Theorem:OptCond} proves the asymptotic mean-square convergence of the algorithm to a critical point.
\begin{theorem}\label{Theorem:Consensus}
  Consider the event-triggered SGD algorithm \eqref{Eq:DSG} under Assumptions~1-7. Then, there holds:
  \begin{align}\label{Eq:MSbound1a}
      \mathbb{E} \left[ \| \tilde{\mathbf{w}}_k \|_2^2 \right] = O\left(\displaystyle\frac{1}{(k+1)^{\delta_2}}\right).
  \end{align}
\end{theorem}
%
Proof of Theorem~\ref{Theorem:Consensus} is given in supplementary material. Define
\begin{equation}
    \gamma_k = \frac{\alpha_k}{\beta_k} = \frac{a/b}{(k+1)^{\delta_2-\delta_1}}.
\end{equation}
Now define a non-negative function $V(\gamma_k,\mathbf{w}_k)$ as
\begin{equation}\label{Vk}
  V(\gamma_k,\mathbf{w}_k)
  = F(\mathbf{w}_k) + \frac{1}{2\gamma_k}\, {\mathbf{w}}^\top_k \left(\mathcal{L}\otimes I_{d_w}\right){\mathbf{w}}_k.
\end{equation}
Taking the gradient with respect to $\mathbf{w}_k$ yields
\begin{equation}\label{dVk}
  \nabla V(\gamma_k,\mathbf{w}_k) = \nabla F(\mathbf{w}_k) + \frac{1}{\gamma_k}\, \left(\mathcal{L}\otimes I_{d_w}\right){\mathbf{w}}_k.
\end{equation}
\begin{theorem}\label{Theorem:InfSum}
  Consider the distributed event-triggered SGD algorithm \eqref{Eq:DSG} under Assumptions~1-7. Then, for the gradient $\nabla V(\gamma_k,\mathbf{w}_k)$ given in \eqref{dVk}, there holds:
  \begin{align}\label{Eqn:SummableGrad}
      \sum\limits_{k=0}^{\infty} \, \alpha_k \mathbb{E}\left[ \left\|  \nabla V(\gamma_k,\mathbf{w}_k)  \right\|^2_2 \right] < \infty.
  \end{align}
\end{theorem}
%
%
\begin{theorem}\label{Theorem:Convergence}
  For the distributed event-triggered SGD algorithm \eqref{Eq:DSG} under Assumptions~1-7, we have
      \begin{align}
      \sum\limits_{k=0}^{\infty} \, \mathbb{E} \left[ \left\| \mathbf{w}_{k + 1} - \mathbf{w}_{k} \right\|_2^2 \right] < \infty\quad\textnormal{and}\label{InEq:DSG3}\\
      \lim_{k\rightarrow\infty} \, \mathbb{E} \left[ \left\| {\mathbf{w}}_{k + 1} - {\mathbf{w}}_{k} \right\|_2^2 \right] =0.\label{ConvergenceEqn}
    \end{align}
\end{theorem}
%
%
See supplementary material section for the proof of Theorem~\ref{Theorem:Convergence}. Define $\bar{\mathbf{w}}_{k} = \frac{1}{n} \left(\mathbf{1}_n \mathbf{1}_n^\top \otimes I_{d_w} \right) \mathbf{w}_{k}$ and $ \overline{\nabla F} (\mathbf{w}_k) = \frac{1}{n} \left(\mathbf{1}_n \mathbf{1}_n^\top \otimes I_{d_w} \right) \nabla F(\mathbf{w}_k)$. Note that $\| \overline{\nabla F} (\mathbf{w}_k) \|_2^2$ $=$ $\frac{1}{n}\| \left( \mathbf{1}_n^\top \otimes I_{d_w} \right) \nabla F(\mathbf{w}_k)\|_2^2$ $=$ $\frac{1}{n}\| \sum_{i=1}^n \nabla f_i(\bm{w}_i(k))\|_2^2$.
\begin{theorem}\label{Theorem:SummableGrad}
  For the distributed event-triggered SGD algorithm \eqref{Eq:DSG} under Assumptions~1-7, we have
  \begin{equation}\label{Eqn:SummableDFbar}
    \sum\limits_{k=0}^{\infty} \, \alpha_k\, \mathbb{E}\left[ \left\| \overline{\nabla F} (\mathbf{w}_k) \right\|^2_2 \right] < \infty.
  \end{equation}
\end{theorem}
%
%
Theorem \ref{Theorem:SummableGrad} establishes results about the weighted sum of expected average gradient norm and the key takeaway from this result is that, for the distributed SGD in \eqref{Eq:DSG1} or \eqref{Eq:DSG} with appropriate step-sizes, the expected average gradient norms cannot stay bounded away from zero (See Theorem 9 of \citep{Bottou2018SIAM}), i.e., $\liminf_{k\rightarrow\infty}\,\mathbb{E}\left[ \left\| \overline{\nabla F} (\mathbf{w}_k) \right\|^2_2 \right] = 0$ or equivalently $\liminf_{k\rightarrow\infty}\,\mathbb{E}\left[ \left\| \sum_{i=1}^n \nabla f_i(\bm{w}_i(k)) \right\|^2_2 \right] = 0$. The rate of such weak convergence results can be obtained as shown in Theorem~\ref{Theorem:ConvergenceRate}.
\begin{theorem}\label{Theorem:ConvergenceRate}
 Let $\{ \mathbf{w}_k\}_{k=0}^{K}$ be generated according to the distributed event-triggered SGD given in \eqref{Eq:DSG} under Assumptions~1-7. Then for $\delta_2 = 1$ we have
 \begin{align}
     \mathbb{E}\left[ \left\| \sum_{i=1}^n \nabla f_i(\,\bm{z}^{\tiny{K}}_i\,) \right\|_2^2 \right]  &= O \left(\frac{1}{\log(K+1)}\right)
 \end{align}
 and for $\delta_2 \in (0.5, \,\, 1)$ we have
\begin{align}
    \mathbb{E}\left[ \left\| \sum_{i=1}^n \nabla f_i(\,\bm{z}^{\tiny{K}}_i\,) \right\|_2^2 \right]  &= O \left(\frac{1}{(K+1)^{1-\delta_2}}\right).
\end{align}
Here $\mathbf{z}^{\tiny{K}} \triangleq \begin{bmatrix} (\bm{z}^{\tiny{K}}_1)^\top & \ldots & (\bm{z}^{\tiny{K}}_n)^\top \end{bmatrix}^\top$ is a random sample from $\{ \mathbf{w}_k\}_{k=0}^{K}$ with probability $\mathbb{P}\left( \mathbf{z}^{\tiny{K}} = \mathbf{w}_k\right) = \frac{\alpha_k}{\sum_{j=0}^{K} \, \alpha_j\,}$. 
\end{theorem}

Finally, we present the following result to illustrate that stronger convergence results follows from the continuity assumption on the Hessian, which has not been utilized in our analysis so far.
\begin{assumption}\label{Assump:LipzHess}
  The Hessians $\nabla^2 f_i(\,\cdot\,)$ $:\mathbb{R}^{d_w}\mapsto \mathbb{R}^{d_w\times d_w}$ are Lipschitz continuous with Lipschitz constants $L_{H_i}$, i.e., $\forall \,\bm{w}_a,\,\bm{w}_b\in\mathbb{R}^{d_w}, \, i=1,\ldots,n$, we have
\begin{equation}\label{Eq:LipzHess}
  \| \nabla^2 f_i(\,\bm{w}_a\,) - \nabla^2 f_i(\,\bm{w}_b\,) \|_2 \leq L_{H_i} \|\bm{w}_a-\bm{w}_b\|_2.
\end{equation}
\end{assumption}
It follows from Assumption~\ref{Assump:LipzHess} that the Hessian $\nabla^2 F(\cdot)$ is Lipschitz continuous, i.e., $\forall \,\mathbf{w}_a,\,\mathbf{w}_b\in\mathbb{R}^{n d_w}$,
\begin{equation}\label{Eq:LipzHess1}
  \| \nabla^2 F(\,\mathbf{w}_a\,) - \nabla^2 F(\,\mathbf{w}_b\,) \|_2 \leq L_H \|\mathbf{w}_a-\mathbf{w}_b\|_2,
\end{equation}
with constant $L_H = \max\limits_i\{L_{H_i}\}$.
\begin{theorem}\label{Theorem:OptCond}
  For the distributed SGD algorithm \eqref{Eq:DSG} under Assumptions 1-8 we have
    \begin{align}
        \lim_{k\rightarrow\infty}\, \mathbb{E}\left[\, \left\| \overline{\nabla F} (\mathbf{w}_k) \right\|^2_2 \,\right] = 0\quad\textnormal{and}\label{InEq:Opt}\\
        \lim_{k\rightarrow\infty}\, \mathbb{E}\left[\, \left\| \sum_{i=1}^n \nabla f_i(\bm{w}_i(k)) \right\|^2_2 \,\right] = 0.
    \end{align}
\end{theorem}
%
%

\begin{figure*}[ht]
  \begin{centering}
      {\includegraphics[width=.93\textwidth]{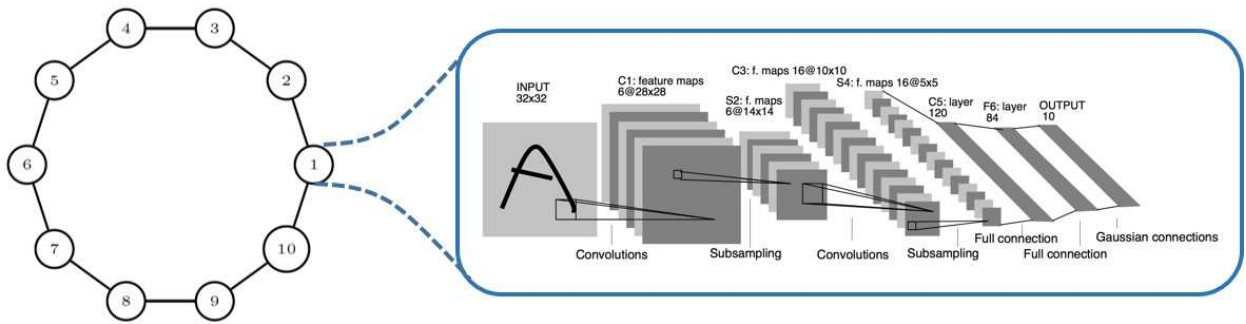}
      \caption{Network of 10 agents, each with its own LeNet-5~\citep{lecun1998gradient} \label{Fig:NNPic}}}
  \end{centering}
\end{figure*}

Similar to the centralized SGD~\citep{Bottou2018SIAM}, the analysis given here shows the mean-square convergence of the distributed algorithm to a critical point, which include the saddle points. Though SGD has shown to escape saddle points efficiently~\citep{Lee2017,2019arXiv190200247F,2019arXiv190204811J}, extension of such results for distributed SGD is currently nonexistent and is a topic for future research.

\section{Application to distributed supervised learning}\label{simulation}

We apply the proposed algorithm for distributedly training 10 neural network agents to recognize handwritten digits in images. Specifically, we use the MNIST\footnote{http://yann.lecun.com/exdb/mnist/} data set containing 60000 images of 10 digits (0-9) for training and 10000 images are used for testing. The 10 agents are connected in an undirected unweighted ring topology as shown in Figure~\ref{Fig:NNPic}. The 10-node ring was selected only since it is one of the least connected network (besides the path) and MNIST contains 10 classes. Proposed algorithm would work for any undirected graph as along as it is connected. 

Each agent aims to train its own neural network, which is a randomly initialized LeNet-5~\citep{lecun1998gradient}. During training, each agent broadcasts its weights to its neighbors at every iteration or aperiodically as described in the proposed algorithm. Here we conduct the following five experiments: (i) Centralized SGD, where a centralized version of the SGD is implemented by a central node having access to all 60000 training images from all classes; (ii) Distributed SGD-r, where all the agents broadcast their respective weights at every iteration, and each agent has access to 6000 training images, \textit{randomly} sampled from the entire training set, which forms the i.i.d. case; (iii) Distributed SGD-s, where all the agents broadcast their weights at every iteration, and each agent has access to the images corresponding to a \textit{single} class, which forms the non-i.i.d. case; (iv) DETSGRAD-r, where the agents aperiodically broadcast their weights using the triggering mechanism in \eqref{TriggCond}, and each agent has access to 6000 training images, \textit{randomly} sampled from the entire training set, i.e., i.i.d. case; (v) DETSGRAD-s, where the agents aperiodically broadcast their weights using the triggering mechanism in \eqref{TriggCond}, and each agent has access to the images corresponding to a \textit{single} class, i.e., non-i.i.d. case. In the single class case, for ease of programming, we set the number of training images available for each agent to 5421 (the minimum number of training images available in a single class, which is digit 5 in MNIST data set). Here we select $\alpha_k = \frac{0.1}{(\varepsilon k + 1)}$ and $\beta_k = \frac{0.2525}{(\varepsilon k + 1)^{1/10}}$, where $\varepsilon = 10^{-5}$ for Distributed SGD and DETSGRAD. We select $\alpha_k = \frac{0.001}{(\varepsilon k + 1)}$ for centralized SGD. Note that using a scale factor $\varepsilon$ does not affect the theoretical results provided in the previous sections. For the DETSGRAD experiments, we select the broadcast event trigger threshold $\upsilon_0\ = 0.2 \times N_{parameters}$, where $N_{parameters}$ is the total number of parameters in each neural network.
\begin{table*}[h!]
\begin{center}
\resizebox{.75\textwidth}{!}{
\begin{tabular}{@{}ccccccccccc@{}} 
 \toprule
 Agent & 1 & 2 & 3 & 4 & 5 & 6 & 7 & 8 & 9 & 10 \\ 
 \midrule
 Dist. SGD-r & 98.97 & 98.97 & 98.97 & 98.97 & 98.97 & 98.97 & 98.97 & 98.97 & 98.97 & 98.97 \\ 
 \midrule
 Dist. SGD-s & 98.86 & 98.86 & 98.86 & 98.87 & 98.86 & 98.86 & 98.86 & 98.87 & 98.85 & 98.87 \\
 \midrule
 DETSGRAD-r & 98.34 & 98.35 & 98.32 & 98.27 & 98.31 & 98.31 & 98.38 & 98.29 & 98.23 & 98.33 \\ 
 \midrule
 DETSGRAD-s & 98.46 & 98.49 & 98.49 & 98.51 & 98.5 & 98.45 & 98.13 & 98.49 & 98.42 & 98.51 \\
 \bottomrule
\end{tabular}}
\caption{Final classification accuracies (\%) of the 10 agents after 40 epochs (240000 iterations for the random sampling/i.i.d. case and 216840 iterations for the single class/non-i.i.d case) using different algorithms. The final accuracy of a single agent using centralized SGD after 10 epochs (600000 iterations) is 98.63\%.}
\label{acctab}
\end{center}
\end{table*}

\begin{table*}[h]
\begin{center}
\resizebox{0.75\textwidth}{!}{
\begin{tabular}{@{}ccccccccccc@{}} 
 \toprule
 Agent & 1 & 2 & 3 & 4 & 5 & 6 & 7 & 8 & 9 & 10 \\ 
 \midrule
 DETSGRAD-r & 61759 & 61455 & 61504 & 61636 & 61738 & 61822 & 61746 & 61712 & 61850 & 61795 \\ 
 \midrule
 DETSGRAD-s & 71756 & 71718 & 71762 & 71983 & 71976 & 71773 & 71762 & 72159 & 72233 & 72208 \\
 \bottomrule
\end{tabular}}
\caption{Total number of event-triggered broadcast events for the 10 agents after 40 epochs. The total number of continuous broadcast events for each agent after 40 epochs is 240000 in the random sampling case, and 216840 in the single class case.}
\label{neventtab}
\end{center}
\end{table*}

The plots of the empirical risk vs. the iterations (parameter update steps), illustrated in Figure~\ref{lossplots} (see supplementary material), show the convergence of the proposed algorithm. The final test accuracies of the 10 agents after 40 training epochs using different algorithms and different training settings are shown in Table~\ref{acctab}. Results obtained here indicate that regardless of how the data are distributed (random or single class), the agents are able to train their network and the distributedly trained networks are able to yield similar performance as that of a centrally trained network. More importantly, in the single class case, agents were able to recognize images from all 10 classes even though they had access to data corresponding only to a single class during the training phase. This result has numerous implications for the machine learning community, specifically for federated multi-task learning under information flow constraints. 

The total number of event-triggered parameter broadcast events for the 10 agents using the DETSGRAD algorithm are shown in Table~\ref{neventtab}. In the random sampling case, by employing broadcast event-triggering mechanism, we are able to reduce the inter-agent communications from 240000 to an average of 61702 over 40 epochs leading to a reduction of 74.2\% in network communications. In the single class case, the agents broadcast the parameters continuously for the first 4 epochs, after which the event-trigger mechanism is started. Here, we are able to reduce the parameter broadcasts for each agent from 216840 to an average of 71933 over 40 epochs leading to a reduction of 66.8\% in network communications. Yet, as can be seen in Table~\ref{acctab}, DETSGRAD gives similar classification performance as distributed SGD with continuous parameter sharing with significant reduction in network communications. The fractions of the broadcast events for the 10 agents over 40 epochs are presented in Figure~\ref{pbplots} (see supplementary material). As expected, the number of broadcast events reduces with the increase in epoch number as the agents converge to the critical point of the empirical risk function. 


\section{Conclusion}\label{sec:conclusion}

This paper presented the development of a distributed stochastic gradient descent algorithm with event-triggered communication mechanism for solving non-convex optimization problems. We presented a novel communication triggering mechanism, which allowed the agents to decidedly reduce the communication overhead by communicating only when the local model has significantly changed from previously communicated model. We presented the sufficient conditions on algorithm step-sizes to guarantee asymptotic mean-square convergence of the proposed algorithm to a critical point and provided the convergence rate of the proposed algorithm. We applied the developed algorithm to a distributed supervised-learning problem, in which a set of 10 networked agents collaboratively train their individual neural nets to recognize handwritten digits in images. Results indicate that regardless of how the data are distributed, the agents are able to train their neural network and the distributedly trained networks are able to yield similar performance to that of a centrally trained network. Numerical results also show that the proposed event-triggered communication mechanism significantly reduced the inter-agent communication wile yielding similar performance to that of a distributedly trained network with constant communication.

\small
\bibliographystyle{aaai}
\bibliography{Biblio}

\newpage
\onecolumn
\normalsize

\beginsupplement

\setlength{\abovedisplayskip}{8pt}
\setlength{\belowdisplayskip}{8pt}

\begin{center}
{\LARGE{\bf Distributed Deep Learning with Event-Triggered Communication}}
\end{center}
\vspace{3pt}
\begin{center}{\Large{\bf(Supplementary Material)}}
\end{center}
\vspace{5pt}

Pseudo-code of the proposed distributed event-triggered SGD is given in Algorithm~\ref{Algorithm1}.

\begin{algorithm}
 \caption{DETSGRAD algorithm}
 \label{Algorithm1}
 \begin{algorithmic}[2]
  \STATE \textit{Input} : $a$, $b$, $\upsilon_0$, $\delta_1$ and $\delta_2$
  \\ \textit{Initialization} : $\mathbf{w}(0) = \begin{bmatrix} {\bm{w}}_1^\top(0) & \ldots & {\bm{w}}_n^\top(0)\end{bmatrix}^\top$
  \FOR {Agent $i = 1$ to $n$}
  \STATE \textit{Sample  } $\bm{\xi}_i(0)$    \,\,\,\&\,\,\,   \textit{compute } ${g}_i\left( \bm{w}_i(0), \bm{\xi}_i(0) \right)$ \label{Step04}
  \STATE \textit{Send  } $\bm{w}_i(0)$  \,\,\,\&\,\,\,  \textit{let \,\,}$\hat{\bm{w}}_i^{(i)} = \bm{w}_i(0)$
  \STATE \textit{Receive  } $\bm{w}_j(0)$  \,\,\,\&\,\,\, \textit{let } $\hat{\bm{w}}_j^{(i)} = \bm{w}_j(0)$,  $\quad\forall\,j\in\mathcal{N}_i$
  \STATE \textit{Update  } $\bm{w}_i(1) = \bm{w}_i(0) - \beta_0 \, \displaystyle\sum_{j\in\mathcal{N}_i}\,a_{ij} \left( \hat{\bm{w}}_i^{(i)} - \hat{\bm{w}}_j^{(i)} \right) - \alpha_0 \,\mathbf{ g}_i\left( \bm{w}_i(0), \bm{\xi}_i(0) \right)$
  \ENDFOR
  \FOR {Iteration $k \geq 1$}
  \FOR {Agent $i = 1$ to $n$}
  \STATE \textit{Sample  } $\bm{\xi}_i(k)$ \,\,\,\&\,\,\, \textit{compute } ${g}_i\left( \bm{w}_i(k), \bm{\xi}_i(k) \right)$ \label{Step04}
  \STATE \textit{Compute  } $\bm{e}_i(k) = {\bm{w}}_i(k)-\hat{\bm{w}}_i^{(i)}$
  \IF {$\left\|\bm{e}_i(k)\right\|_1 \geq \upsilon_0\,\alpha_k$}
  \STATE \textit{Send  } $\bm{w}_i(k)$ \,\,\,\&\,\,\, \textit{let \,} $\hat{\bm{w}}_i^{(i)} = \bm{w}_i(k)$
  \ENDIF
  \IF {any $\bm{w}_j(k)$ received}
  \STATE \textit{Let }$\hat{\bm{w}}_j^{(i)} = \bm{w}_j(k)$
  \ENDIF
  \STATE \textit{Update  } $\bm{w}_i(k+1) = \bm{w}_i(k) - \beta_k \, \displaystyle\sum_{j\in\mathcal{N}_i}\,a_{ij} \left( \hat{\bm{w}}_i^{(i)} - \hat{\bm{w}}_j^{(i)} \right) - \alpha_k \,\mathbf{ g}_i\left( \bm{w}_i(k), \bm{\xi}_i(k) \right)$
  \ENDFOR
  \ENDFOR
 \end{algorithmic}
 \end{algorithm}
\vspace{2pt}
\noindent Note that here $\alpha_k$ and $\beta_k$ are defined as
\begin{align}
\alpha_k = \frac{a}{(k+1)^{\delta_2}}
\end{align}
and
\begin{align}
\beta_k = \frac{b}{(k+1)^{\delta_1}}.
\end{align}

\section{Additional Numerical Results}

Numerical results from the implementation of the proposed algorithm for distributed supervised learning are presented here.
\begin{figure*}[h!]
  \begin{centering}
      \subfigure[Centralized SGD]{
      \includegraphics[scale=0.35]{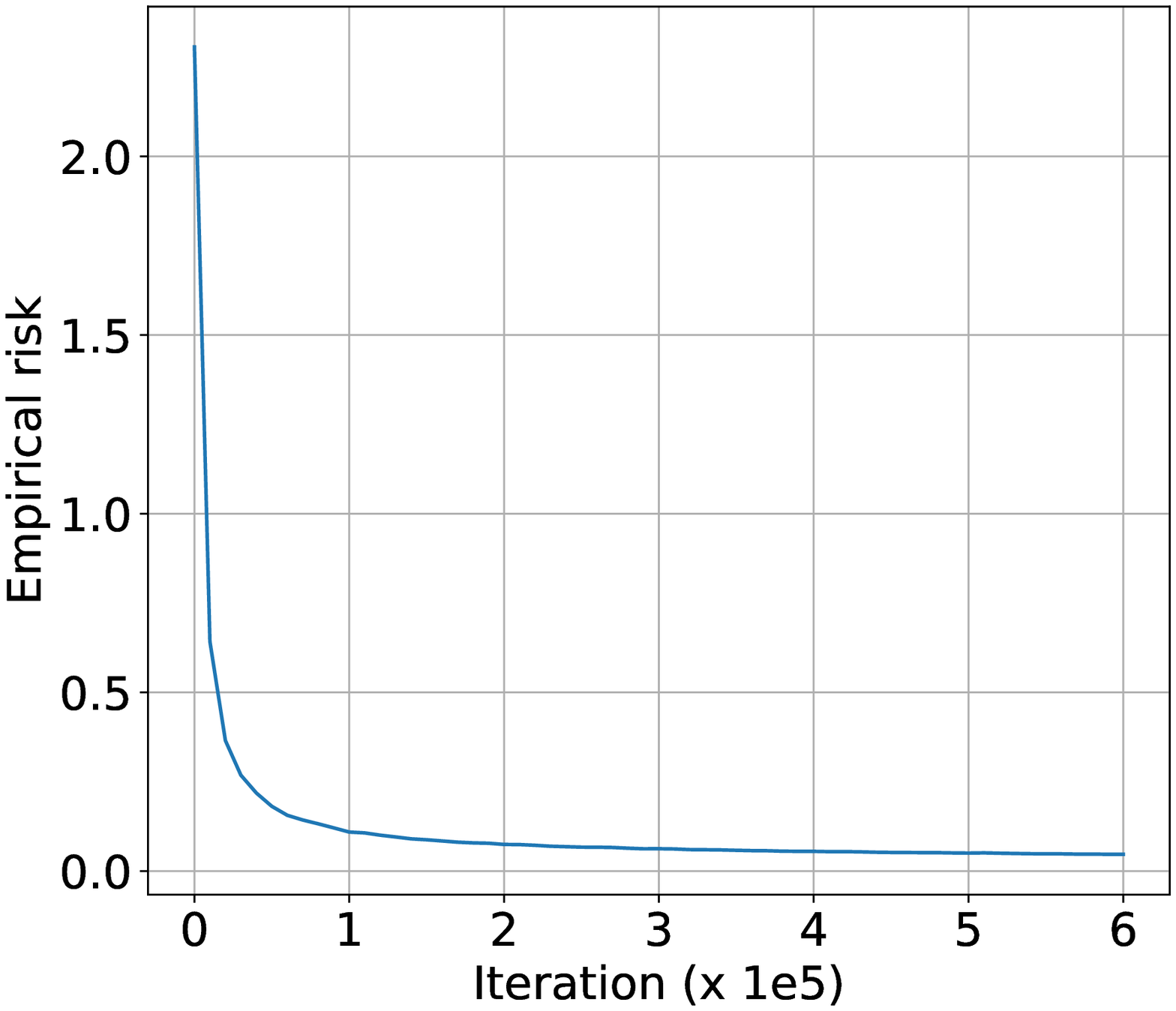}\label{Fig:Cent}}
      \subfigure[Distributed SGD-r]{
      \includegraphics[scale=0.35]{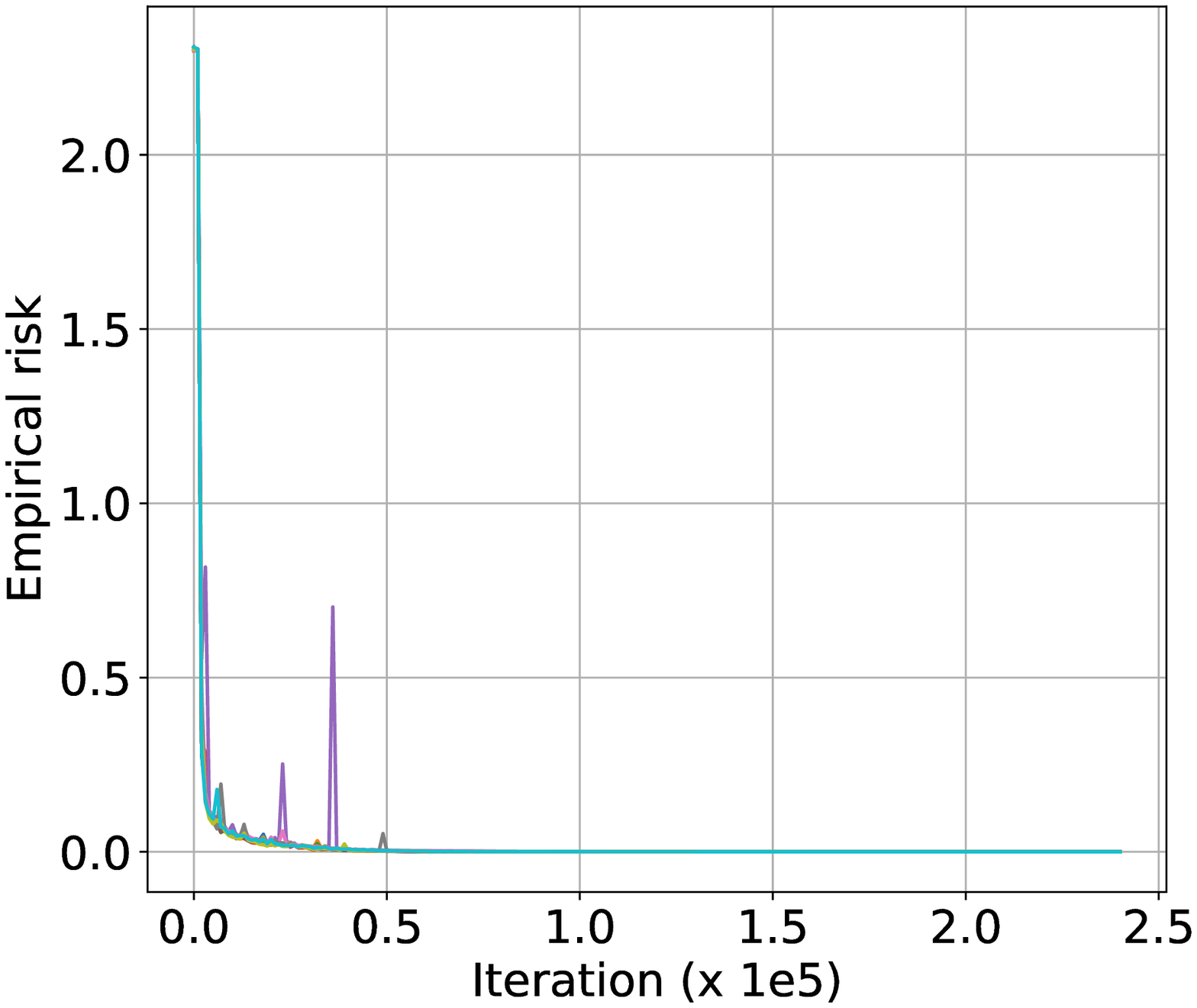}\label{Fig:Distacl}}
      \subfigure[Distributed SGD-s]{
      \includegraphics[scale=0.35]{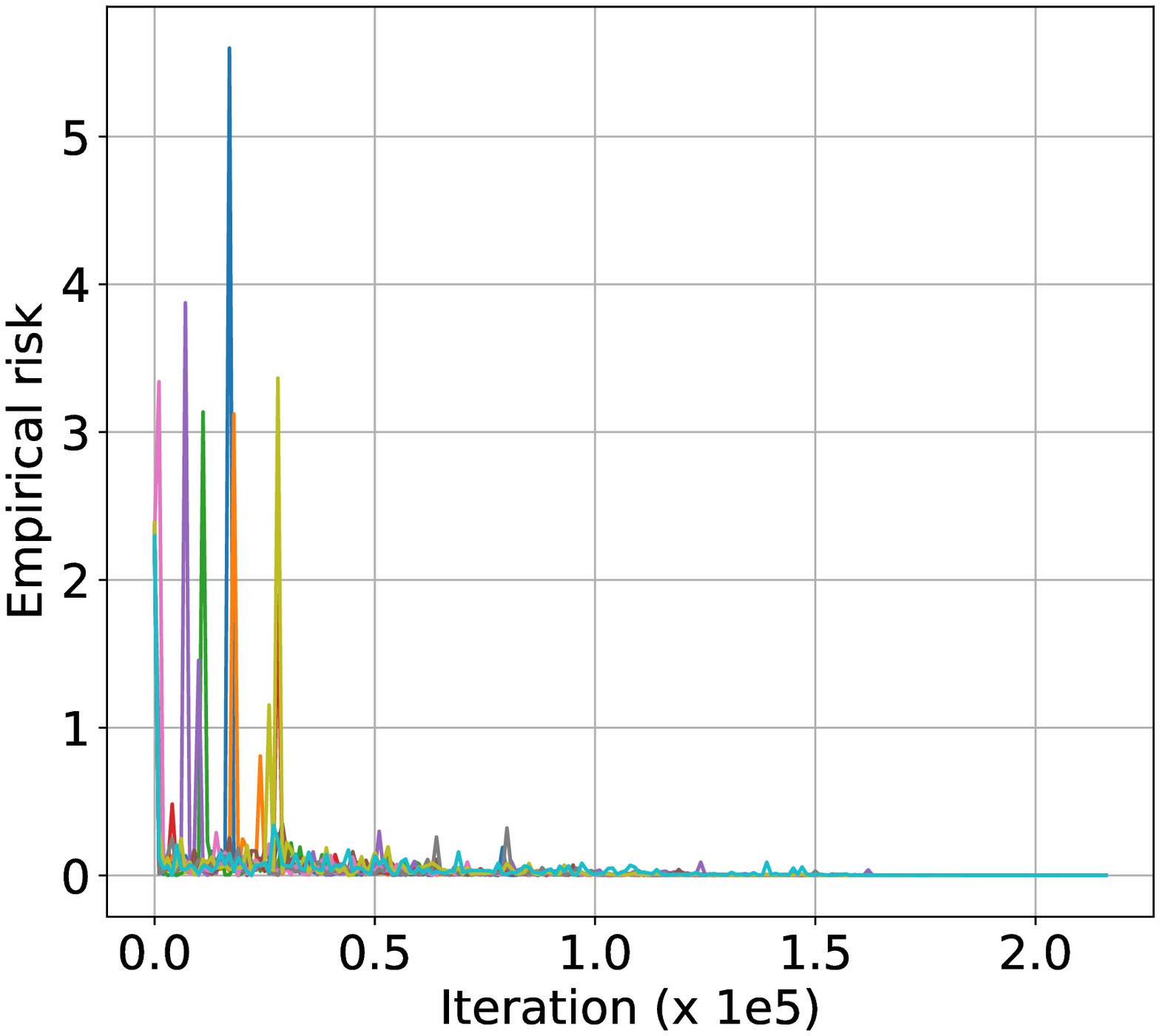}\label{Fig:Distscl}}
      \subfigure[DETSGRAD-r]{
      \includegraphics[scale=0.35]{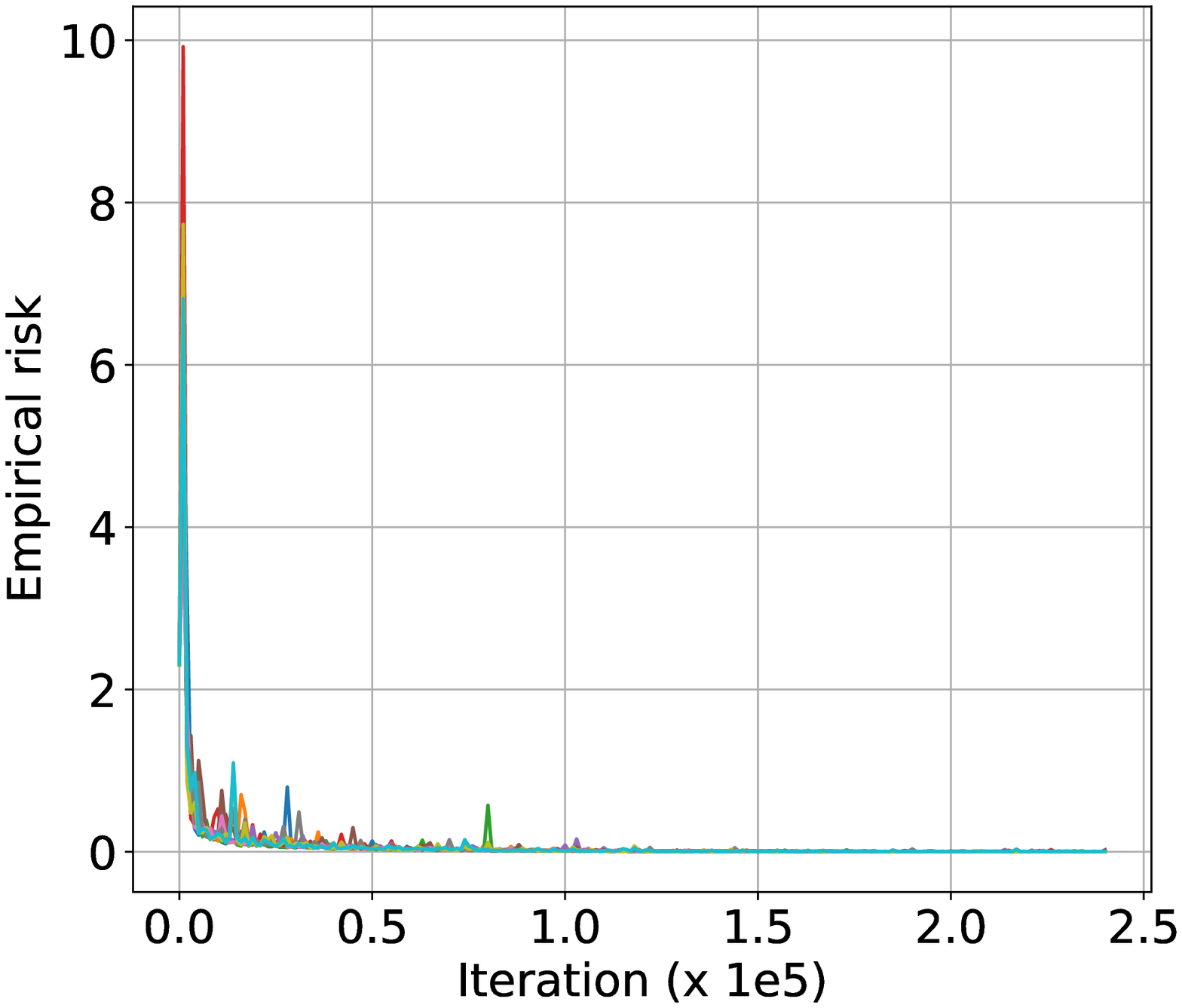}\label{Fig:ADistacl}}
      \subfigure[DETSGRAD-s]{
      \includegraphics[scale=0.35]{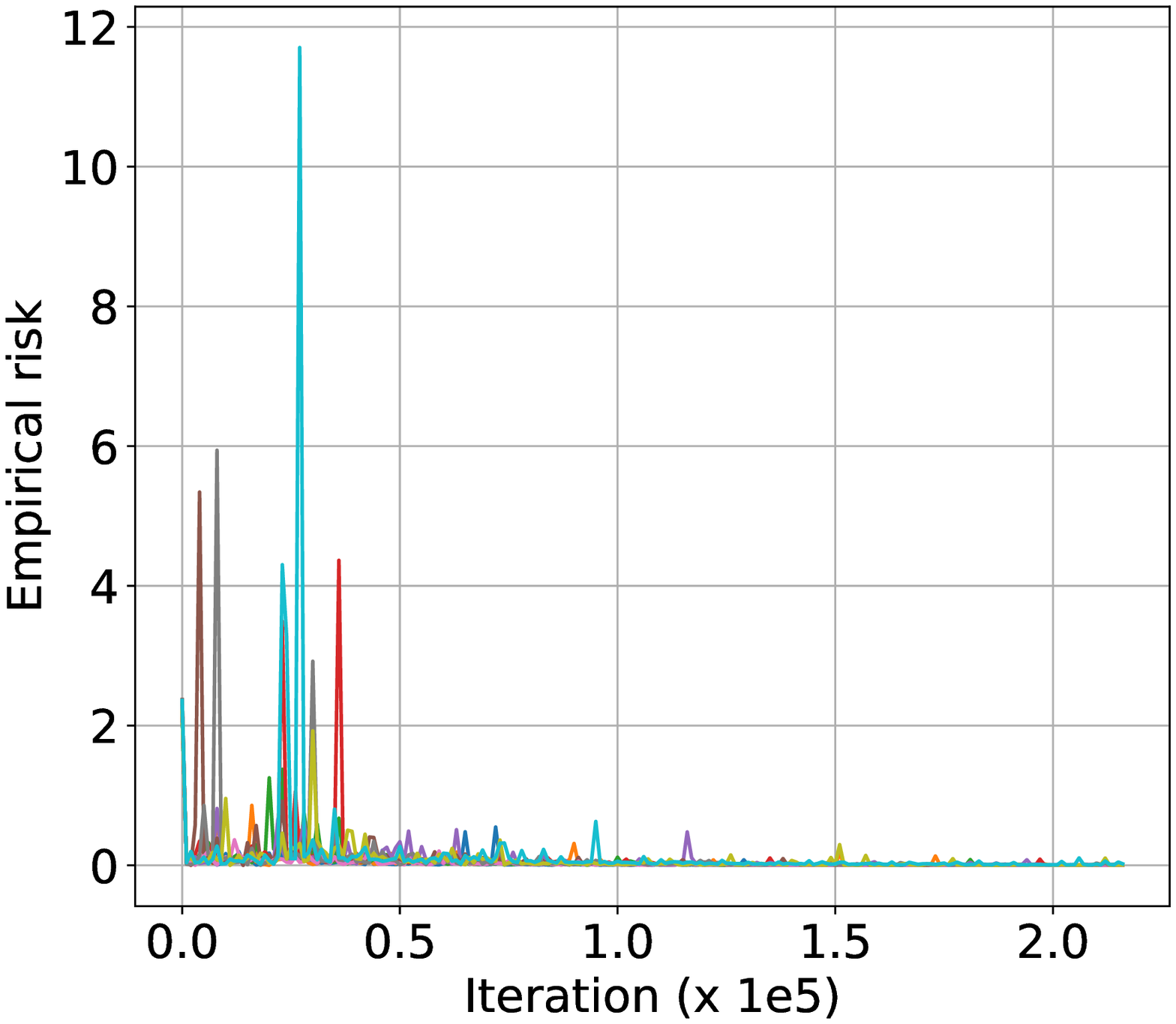}\label{Fig:ADistscl}}
      \caption{Empirical risk for all five experiments.}
      \label{lossplots}
  \end{centering}
\end{figure*}

\begin{figure*}[h!]
  \begin{centering}
      \subfigure[DETSGRAD-r]{
      \includegraphics[scale=0.35]{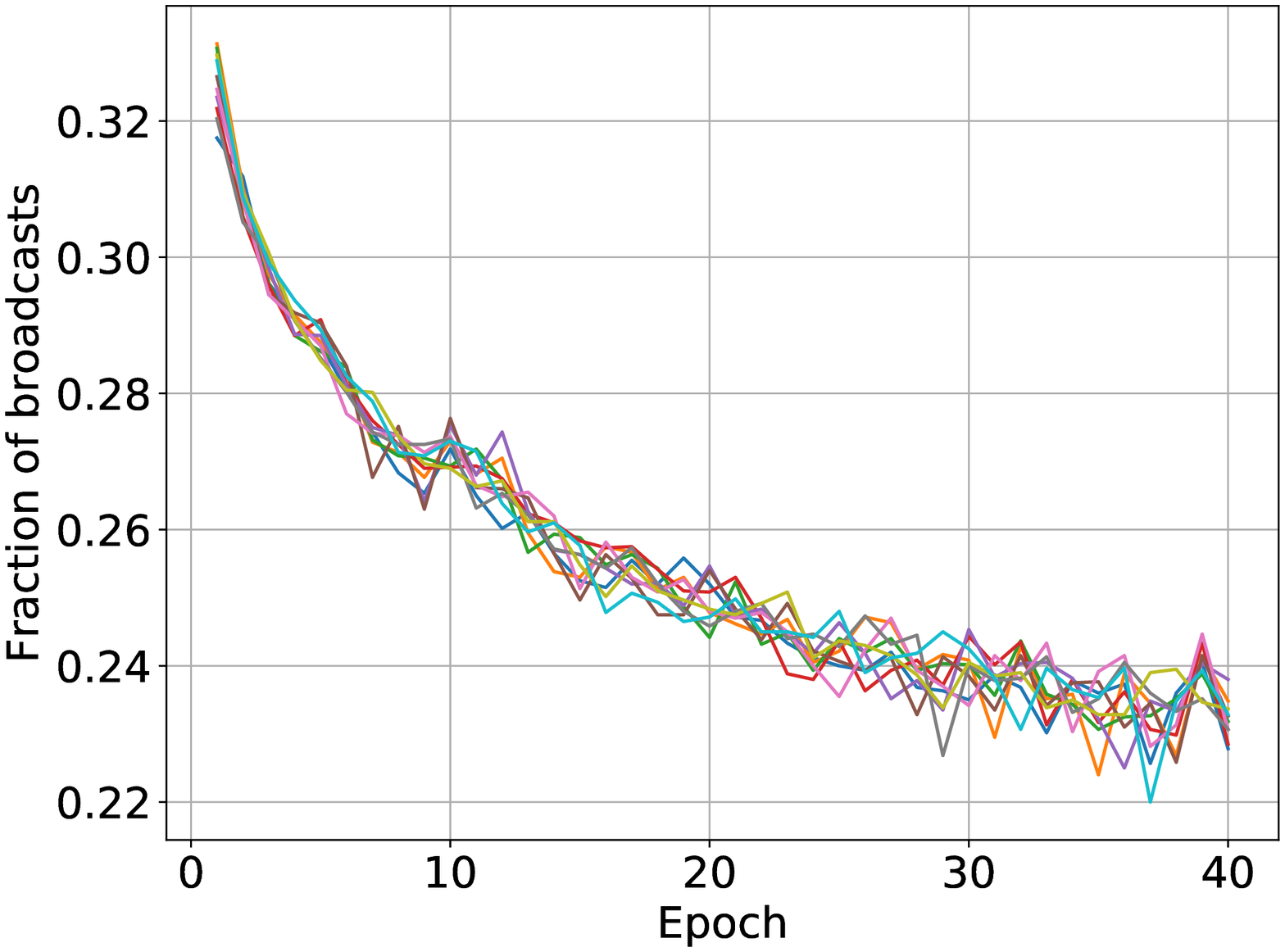}\label{Fig:PBacl}}
      \subfigure[DETSGRAD-s]{
      \includegraphics[scale=0.35]{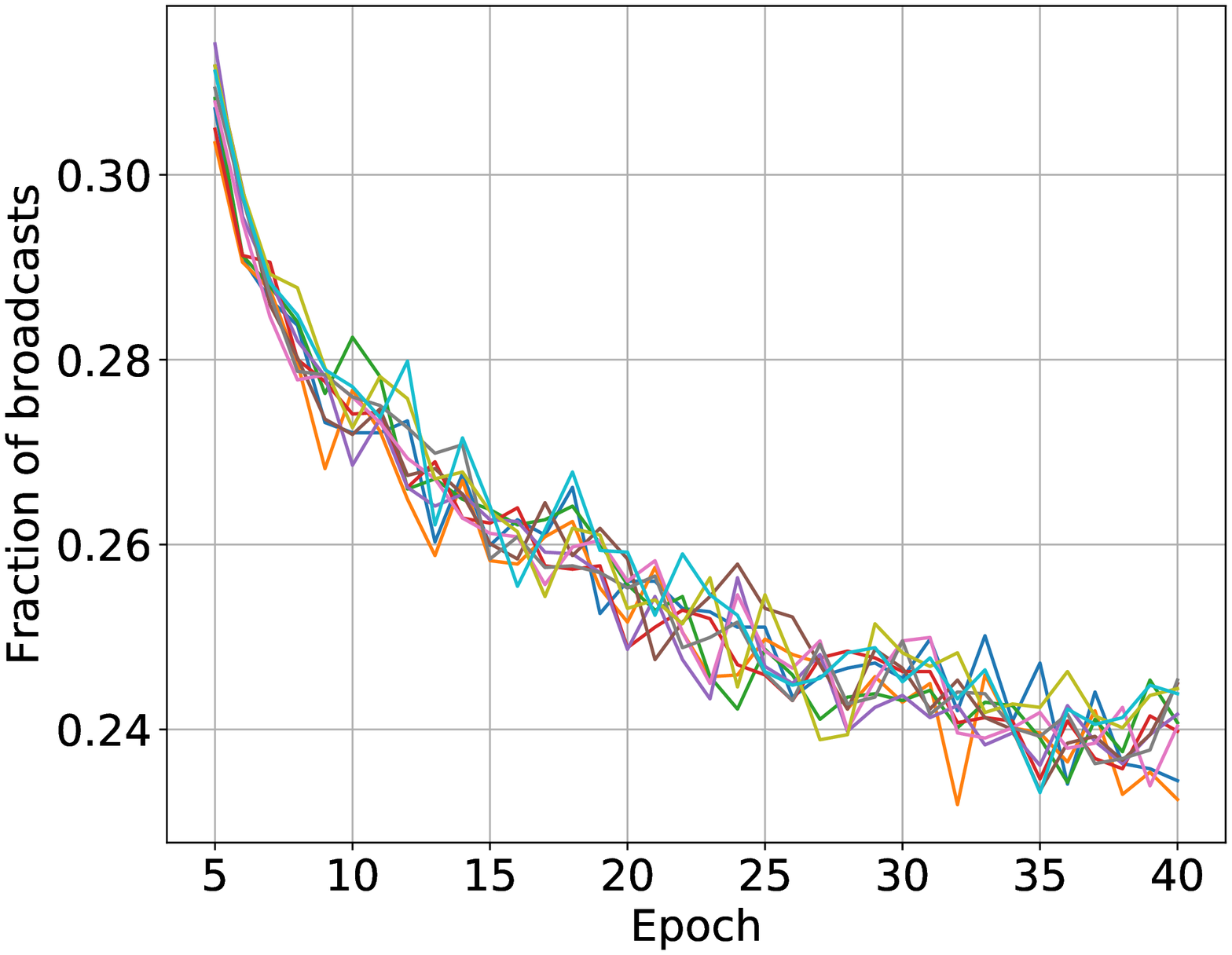}\label{Fig:PBscl}}
      \caption{Fraction of event-triggered broadcast events for the 10 agents compared to continuous broadcasting case.}
      \label{pbplots}
  \end{centering}
\end{figure*}

Detailed proofs of the theoretical results given in the main body of the manuscript are given next, but first we present few useful Lemmas.

\section{Few Useful Lemmas}\label{Appdx:A}

\begin{lemma}\label{Lemma:Kar}
  Let $\{z_k\}$ be a non-negative sequence satisfying
\begin{align}\label{Eqn:lamma2}
z_{k+1} \leq \left( 1 - r_1(k)\right)z_k + r_2(k),
\end{align}
where $\{r_1(k)\}$ and $\{r_2(k)\}$ are sequences with
\begin{equation}
    \frac{a_1}{(k+1)^{\epsilon_1}} \leq r_1(k) \leq 1 \quad \textnormal{and} \quad r_2(k) \leq \frac{a_2}{(k+1)^{\epsilon_2}},
\end{equation}
where $0 < a_1$, $0 < a_2$, $0 \leq \epsilon_1 < 1$, and $\epsilon_1 < \epsilon_2$. Then $ (k+1)^{\epsilon_0}z_k \rightarrow 0$ as $k\rightarrow \infty$ for all $0 \leq \epsilon_0 < \epsilon_2 - \epsilon_1$.
\end{lemma}
\begin{proof}
This Lemma follows directly from Lemma 4.1 of \citep{Kar2013SIAM}.
\end{proof}

\begin{lemma}\label{Lemma:Robbins}
  Let $\{v_k\}$ be a non-negative sequence for which the following relation hold for all $k \geq 0$:
  \begin{equation}\label{Eq:Robbins}
      v_{k+1} \leq (1 + a_k) v_k - u_k + w_k,
  \end{equation}
  where $a_k\geq 0$, $u_k\geq 0$ and $w_k\geq 0$ with $\sum\limits_{k=0}^{\infty}a_k < \infty $ and $\sum\limits_{k=0}^{\infty}w_k < \infty $. Then the sequence $\{v_k\}$ will converge to $v\geq0$ and we further have $\sum\limits_{k=0}^{\infty}u_k < \infty $.
\end{lemma}
\begin{proof}
See \citep{ROBBINS1971233}.
\end{proof}
\begin{lemma}\label{Lemma10}
  Let $\gamma_k \triangleq \displaystyle\frac{a/b}{(k+1)^{\epsilon}}$ with $0 < \epsilon \leq 1$. Then it holds
  \begin{equation}
      \gamma_{k+1}^{-1} - \gamma_{k}^{-1} \leq \frac{2b\epsilon}{a} (k+1)^{\epsilon-1}.
  \end{equation}
  \vspace{1pt}
\end{lemma}
\begin{proof}
First note that $(1+x)^{\epsilon}$ is a monotonically increasing function for all $x\in[0,\,1]$. Thus the following inequality holds for all $ \epsilon\in(0,\,1]$
\begin{equation}\label{Eq:Lemma10}
    (1+x)^{\epsilon} - 1 \leq 2\epsilon x, \quad \forall x\in[0,\,1].
\end{equation}
Note that
\begin{align}
    \gamma_{k+1}^{-1} - \gamma_{k}^{-1} &= \frac{b}{a} \left( (k+2)^{\epsilon} - (k+1)^{\epsilon} \right) = \frac{b}{a} (k+1)^{\epsilon} \left( \left( 1 + \frac{1}{k+1} \right)^{\epsilon} - 1 \right)
\end{align}
From \eqref{Eq:Lemma10} we have
\begin{align}
    \left( 1 + \frac{1}{k+1} \right)^{\epsilon} - 1 \leq 2\epsilon \frac{1}{k+1}
\end{align}
Therefore
\begin{align}
    \gamma_{k+1}^{-1} - \gamma_{k}^{-1} \leq \frac{2b\epsilon}{a} (k+1)^{\epsilon-1}.
\end{align}
\end{proof}

\noindent We have the following result from Assumptions~\ref{Assump:Lipz} and \ref{Assump:Grad2}.
\begin{proposition}\label{Assump:BoundedGrad}
Given Assumptions \ref{Assump:Lipz} and \ref{Assump:Grad2}, there exists a positive constant $\mu_g < \infty$ such that
\begin{align}\label{Eqn:BoundedGrad}
\sup\limits_{k\geq0} \, \mathbb{E} \left[ \| \mathbf{g}(\mathbf{w}_k,\bm{\xi}_k) \|_2^2 \right] \leq \mu_g.
\end{align}
\end{proposition}
\begin{proof}
Lipschitz continuity of $f_i(\,\cdot\,)$ implies the function $F(\cdot)$ is Lipschitz continuous. From Assumption \ref{Assump:Lipz}, $\nabla F(\cdot)$ is Lipschitz continuous. It follows from Lemma 3.3 in~\cite{khalil2002nonlinear} that $\nabla F(\mathbf{w})$ is bounded $\forall\, \mathbf{w}\in\mathbb{R}^{n d_w}$. Now \eqref{Eqn:BoundedGrad} follows from taking the expectation of \eqref{Eq:2ndGrad}.
\end{proof}
Note that the above result \eqref{Eqn:BoundedGrad} is usually just assumed in literature, e.g.~\citep{Nedic2009TAC,Bottou2018SIAM,Tatarenko2017,Zeng2018TSP}.


\section{Proof of Theorem \ref{Theorem:Consensus}}\label{Appdx:B}
From \eqref{Eqn:wTilde} we have
\begin{align}
    \| \tilde{\mathbf{w}}_{k+1} \|_2 \leq \| \left( \left(I_{n}-\beta_k\mathcal{L}\right) \otimes I_{d_w}\right)  \tilde{\mathbf{w}}_k \|_2 &+ \alpha_k \| \left(M \otimes I_{d_w}\right) \|_2 \| \mathbf{g}(\mathbf{w}_k,\bm{\xi}_k) \|_2 + \beta_k \| \left(\mathcal{L} \otimes I_{d_w}\right) \|_2 \| \mathbf{e}(k) \|_2.
\end{align}
From the triggering condition \eqref{TriggCond} we have $\left\|\mathbf{e}(k)\right\|_1 < n\upsilon_0\alpha_k$. Thus we have
\begin{align}
    \| \tilde{\mathbf{w}}_{k+1} \|_2 \leq \| \left( \left(I_{n}-\beta_k\mathcal{L}\right) \otimes I_{d_w}\right)  \tilde{\mathbf{w}}_k \|_2 &+ \alpha_k \| \left(M \otimes I_{d_w}\right) \|_2 \| \mathbf{g}(\mathbf{w}_k,\bm{\xi}_k) \|_2 + \alpha_k n\upsilon_0 \| \left(\mathcal{L} \otimes I_{d_w}\right) \|_2 .
\end{align}
Since $\mathbf{1}_{nd_w}^\top \tilde{\mathbf{w}}_k = 0$, it follows from Assumption \ref{Assump:Graph} and Lemma 4.4 of \citep{Kar2013SIAM} that
\begin{align}
    \| \left( \left(I_{n}-\beta_k\mathcal{L}\right) \otimes I_{d_w}\right)  \tilde{\mathbf{w}}_k \|_2 \leq (1 - \beta_k\lambda_2(\mathcal{L})) \| \tilde{\mathbf{w}}_k \|_2,
\end{align}
where $\lambda_2(\cdot)$ denotes the second smallest eigenvalue. Thus we have
\begin{align}
    \| \tilde{\mathbf{w}}_{k+1} \|_2
    &\leq  (1 - \beta_k\lambda_2(\mathcal{L})) \| \tilde{\mathbf{w}}_k \|_2 + \alpha_k  \left( \| \mathbf{g}(\mathbf{w}_k,\bm{\xi}_k) \|_2 + n\upsilon_0 \sigma_{\max}(\mathcal{L}) \right),
\end{align}
where $\sigma_{\max}(\cdot)$ denotes the largest singular value. Now we use the following inequality
\begin{equation}\label{InEqTrick}
    (x+y)^2 \leq (1+\theta)x^2 + \left( 1+ \frac{1}{\theta}\right)y^2,
\end{equation}
for all $x,y,\in\mathbb{R}$ and $\theta > 0$. Selecting $\theta = \beta_k\lambda_2(\mathcal{L})$ yields
\begin{align}
    \| \tilde{\mathbf{w}}_{k+1} \|_2^2
    &\leq  (1 + \beta_k\lambda_2(\mathcal{L})) (1 - \beta_k\lambda_2(\mathcal{L}))  \| \tilde{\mathbf{w}}_k \|_2^2
     + \alpha_k^2 \left( 1 + \frac{1}{\beta_k\lambda_2(\mathcal{L})} \right) \left( \| \mathbf{g}(\mathbf{w}_k,\bm{\xi}_k) \|_2 + n\upsilon_0 \sigma_{\max}(\mathcal{L}) \right)^2 \\
    &\leq (1 - \beta^2_k\lambda_2(\mathcal{L})^2)  \| \tilde{\mathbf{w}}_k \|_2^2  + 2\alpha_k^2 \left( \frac{1 + \beta_k\lambda_2(\mathcal{L})}{\beta_k\lambda_2(\mathcal{L})} \right)  \left( \| \mathbf{g}(\mathbf{w}_k,\bm{\xi}_k) \|_2^2 + \left( n\upsilon_0 \sigma_{\max}(\mathcal{L}) \right)^2 \right)
\end{align}
Now taking the expectation yields
\begin{align}\label{wTildeE2}
\begin{split}
   \mathbb{E} \left[ \| \tilde{\mathbf{w}}_{k+1} \|_2^2 \right]
    &\leq   (1 - \beta^2_k\lambda_2(\mathcal{L})^2) \mathbb{E} \left[ \| \tilde{\mathbf{w}}_k \|_2^2 \right]
      + 2\alpha_k^2 \left( \frac{1 + \beta_k\lambda_2(\mathcal{L})}{\beta_k\lambda_2(\mathcal{L})} \right) \left( \mathbb{E} \left[ \| \mathbf{g}(\mathbf{w}_k,\bm{\xi}_k) \|_2^2 \right] + \left( n\upsilon_0 \sigma_{\max}(\mathcal{L}) \right)^2
    \right)
\end{split}
\end{align}
Using Proposition \ref{Assump:BoundedGrad}, \eqref{wTildeE2} can be written as
\begin{align}\label{wTildeE2a}
\begin{split}
   \mathbb{E} \left[ \| \tilde{\mathbf{w}}_{k+1} \|_2^2 \right]
    &\leq (1 - \beta^2\lambda_2(\mathcal{L})^2) \mathbb{E} \left[ \| \tilde{\mathbf{w}}_k \|_2^2 \right]  + 2\frac{\alpha_k^2}{\beta_k} \left( \frac{ \left(1 + \beta_k\lambda_2(\mathcal{L})\right)}{\lambda_2(\mathcal{L})} \right)
    \left( \mu_g + \left( n\upsilon_0 \sigma_{\max}(\mathcal{L}) \right)^2 \right)
\end{split}
\end{align}
Note that
\begin{align}
    2\left( \displaystyle \frac{ \left(1 + \beta_k\lambda_2(\mathcal{L})\right)}{\lambda_2(\mathcal{L})} \right)\left( \mu_g + \left( n\upsilon_0 \sigma_{\max}(\mathcal{L}) \right)^2 \right)
    \leq 2\left( \displaystyle \frac{ \left(1 + b\lambda_2(\mathcal{L})\right)}{\lambda_2(\mathcal{L})} \right) \left( \mu_g + \left( n\upsilon_0 \sigma_{\max}(\mathcal{L}) \right)^2 \right)
\end{align}
and for some constant $\mu_a > 0$, we have
\begin{align}
    2\left( \displaystyle \frac{ \left(1 + b\lambda_2(\mathcal{L})\right)}{\lambda_2(\mathcal{L})} \right) \left( \mu_g + \left( n\upsilon_0 \sigma_{\max}(\mathcal{L}) \right)^2 \right) \leq \mu_a
\end{align}
 Let $r_1(k) = \beta^2_k\lambda_2(\mathcal{L})^2 = \displaystyle\frac{b^2\lambda_2(\mathcal{L})^2}{(k+1)^{2\delta_1}}$ and $r_2(k) = \displaystyle2\frac{\alpha_k^2}{\beta_k} \left( \frac{ \left(1 + \beta_k\lambda_2(\mathcal{L})\right)}{\lambda_2(\mathcal{L})} \right)
    \left( \mu_g + \left( n\upsilon_0 \sigma_{\max}(\mathcal{L}) \right)^2 \right) \leq \frac{a^2\mu_a/b}{(k+1)^{2\delta_2-\delta_1}}$. Now \eqref{wTildeE2a} can be written in the form of \eqref{Eqn:lamma2} with  $\epsilon_1 = 2\delta_1$ and $\epsilon_2 = 2\delta_2-\delta_1$. Thus it follows from Lemma \ref{Lemma:Kar} that
\begin{align}
    (k+1)^{\delta_0}\, \mathbb{E} \left[ \| \tilde{\mathbf{w}}_k \|_2^2 \right] \rightarrow 0
    \quad \textnormal{as}\quad k\rightarrow \infty, \quad \forall\,0 \leq \delta_0 < 2\delta_2-3\delta_1.
\end{align}
Thus there exists a constant $0 < \mu_{w} < \infty$ such that for all $k\geq 0$
\begin{align}\label{Eq:MSbound1}
    \mathbb{E} \left[ \| \tilde{\mathbf{w}}_k \|_2^2 \right] \leq \mu_w \frac{1}{(k+1)^{\delta_0}}, \quad \forall\,0 \leq \delta_0 < 2\delta_2-3\delta_1.
\end{align}
Now \eqref{Eq:MSbound1a} follows from Assumption~\ref{Assump:AlphaBeta} that $\delta_2 > 3\delta_1$.


\section{Proof of Theorem \ref{Theorem:InfSum}}\label{Appdx:C}
From \eqref{dVk} we have
\begin{align}
&\nabla V(\gamma_k,\mathbf{w}_{k+1}) - \nabla V(\gamma_k,\mathbf{w}_k) = \nabla F(\,\mathbf{w}_{k+1}\,) - \nabla F(\,\mathbf{w}_k\,)  + \frac{1}{\gamma_k}  \left(\mathcal{L}\otimes I_{d_w}\right) \left({\mathbf{w}}_{k+1} -  {\mathbf{w}}_k \right)
\end{align}
Now based on Assumption \ref{Assump:Lipz}, for a fixed $\gamma_k$, $\nabla V(\gamma_k,\,\mathbf{w}\,)$ is Lipschitz continuous in $\mathbf{w}$. Thus we have
\begin{align}
\begin{split}
  \| \nabla V(\gamma_k,\mathbf{w}_{k+1}) - &\nabla V(\gamma_k,\mathbf{w}_k)\|_2  \leq  \left( L + \frac{\sigma_{\max}(\mathcal{L})}{\gamma_k} \right) \|\mathbf{w}_{k+1}-\mathbf{w}_k\|_2 \label{Eq:dV00}
\end{split}
\end{align}
It follows from Lemma~\ref{Lemma:Lipz} that
\begin{align}
\begin{split}
V(\gamma_k,\mathbf{w}_{k+1}) - V(\gamma_k,\mathbf{w}_k)
  &\leq \left( \nabla F(\mathbf{w}_k) + \frac{1}{\gamma_k} \left(\mathcal{L}\otimes I_{d_w}\right)\mathbf{w}_k \right)^\top \left( \mathbf{w}_{k + 1} - \mathbf{w}_{k} \right)   \\
  &  \qquad\qquad + \frac{1}{2}\left( L + \frac{\sigma_{\max}(\mathcal{L})}{\gamma_k} \right) \| \mathbf{w}_{k + 1} - \mathbf{w}_{k} \|_2^2
\end{split}\label{Eq:dV}
\end{align}
Note that the distributed SGD algorithm in \eqref{Eq:DSG1a} can be rewritten as
\begin{align}\label{Eq:DSG2}
  \mathbf{w}_{k + 1} - \mathbf{w}_{k} &= - \alpha_k \left( \mathbf{g}(\mathbf{w}_k,\bm{\xi}_k) + \frac{1}{\gamma_k} \left(\mathcal{L}\otimes I_{d_w}\right)\mathbf{w}_k
  - \frac{1}{\gamma_k} \left(\mathcal{L}\otimes I_{d_w}\right)\mathbf{e}_k  \right) .
\end{align}
Substituting \eqref{Eq:DSG2} into \eqref{Eq:dV} and taking the conditional expectation $\mathbb{E}_{\xi}\left[ \,\cdot\, \right]$ yields
\begin{align}
\begin{split}
  &\mathbb{E}_{\xi}\left[V(\gamma_k,\mathbf{w}_{k+1})\right] - V(\gamma_k,\mathbf{w}_k) \leq   - \alpha_k \left( \nabla F(\mathbf{w}_k) + \frac{1}{\gamma_k} \left(\mathcal{L}\otimes I_{d_w}\right)\mathbf{w}_k \right)^\top \\
  &\qquad\qquad\qquad\qquad\qquad\qquad\qquad\qquad\qquad \times\left( \mathbb{E}_{\xi}\left[ \mathbf{g}(\mathbf{w}_k,\bm{\xi}_k) \right] + \frac{1}{\gamma_k} \left(\mathcal{L}\otimes I_{d_w}\right)\mathbf{w}_k - \frac{1}{\gamma_k} \left(\mathcal{L}\otimes I_{d_w}\right)\mathbf{e}_k \right)\\
  &\qquad\quad\qquad\qquad+ \frac{\alpha_k^2}{2}\left( L + \frac{\sigma_{\max}(\mathcal{L})}{\gamma_k} \right) \mathbb{E}_{\xi}\left[ \left\| \mathbf{g}(\mathbf{w}_k,\bm{\xi}_k) + \frac{1}{\gamma_k} \left(\mathcal{L}\otimes I_{d_w}\right)\mathbf{w}_k
  - \frac{1}{\gamma_k} \left(\mathcal{L}\otimes I_{d_w}\right)\mathbf{e}_k \right\|_2^2 \right]  \quad \textnormal{a.s.}
\end{split}\label{Eq:dV0}
\end{align}
Based on Assumption \ref{Assump:Grad1}, there exists $\mu > 0$ such that
\begin{align*}
  \left( \nabla F(\mathbf{w}_k) + \frac{1}{\gamma_k} \left(\mathcal{L}\otimes I_{d_w}\right)\mathbf{w}_k \right)^\top  \bigg(  \mathbb{E}_{\xi}\left[\, \mathbf{g}(\mathbf{w}_k,\bm{\xi}_k) \,\right] +  \left.\frac{1}{\gamma_k} \left(\mathcal{L}\otimes I_{d_w}\right)\mathbf{w}_k \right)
  \geq \mu \left\|  \nabla F(\mathbf{w}_k) + \frac{1}{\gamma_k} \left(\mathcal{L}\otimes I_{d_w}\right)\mathbf{w}_k \right\|^2_2,\,\, \textnormal{a.s.}&
  \end{align*}
Also note that
\begin{align*}
\left\| \mathbf{g}(\mathbf{w}_k,\bm{\xi}_k) + \frac{1}{\gamma_k} \left(\mathcal{L}\otimes I_{d_w}\right)\mathbf{w}_k
  - \frac{1}{\gamma_k} \left(\mathcal{L}\otimes I_{d_w}\right)\mathbf{e}_k \right\|_2^2 \leq
  2\left\| \mathbf{g}(\mathbf{w}_k,\bm{\xi}_k) + \frac{1}{\gamma_k} \left(\mathcal{L}\otimes I_{d_w}\right)\mathbf{w}_k \right\|_2^2
  + 2\left\| \frac{1}{\gamma_k} \left(\mathcal{L}\otimes I_{d_w}\right)\mathbf{e}_k \right\|_2^2&
\end{align*}
Thus we have
\begin{align}
\begin{split}
  &\mathbb{E}_{\xi}\left[V(\gamma_k,\mathbf{w}_{k+1})\right] - V(\gamma_k,\mathbf{w}_k) \leq   \alpha_k^2 \left( L + \frac{\sigma_{\max}(\mathcal{L})}{\gamma_k} \right)
  \mathbb{E}_{\xi}\left[ \| \mathbf{g}(\mathbf{w}_k,\bm{\xi}_k) + \frac{1}{\gamma_k} \left(\mathcal{L}\otimes I_{d_w}\right)\mathbf{w}_k \|_2^2 \right]\\
  & + \alpha_k^2 \left( L + \frac{\sigma_{\max}(\mathcal{L})}{\gamma_k} \right) \left\| \frac{1}{\gamma_k} \left(\mathcal{L}\otimes I_{d_w}\right)\mathbf{e}_k \right\|_2^2 - \alpha_k \mu \left\|  \nabla F(\mathbf{w}_k) + \frac{1}{\gamma_k} \left(\mathcal{L}\otimes I_{d_w}\right)\mathbf{w}_k \right\|^2_2\\
  & + \beta_k \left( \nabla F(\mathbf{w}_k) + \frac{1}{\gamma_k} \left(\mathcal{L}\otimes I_{d_w}\right)\mathbf{w}_k \right)^\top\left(\mathcal{L}\otimes I_{d_w}\right)\mathbf{e}_k \qquad \textnormal{a.s.}
\end{split}\label{Eq:dV1}
\end{align}
Let
\begin{equation}\label{ck}
  c_k \triangleq \left( \alpha_k L + \sigma_{\max}(\mathcal{L})\beta_k \right).
\end{equation}
Now \eqref{Eq:dV1} can be written as
\begin{align}
\begin{split}
  &\mathbb{E}_{\xi}\left[V(\gamma_k,\mathbf{w}_{k+1})\right] - V(\gamma_k,\mathbf{w}_k) \leq   \alpha_k c_k
  \mathbb{E}_{\xi}\left[ \| \mathbf{g}(\mathbf{w}_k,\bm{\xi}_k) + \frac{1}{\gamma_k} \left(\mathcal{L}\otimes I_{d_w}\right)\mathbf{w}_k \|_2^2 \right]\\
  & + \alpha_k c_k \left\| \frac{1}{\gamma_k} \left(\mathcal{L}\otimes I_{d_w}\right)\mathbf{e}_k \right\|_2^2 - \alpha_k \mu \left\|  \nabla F(\mathbf{w}_k) + \frac{1}{\gamma_k} \left(\mathcal{L}\otimes I_{d_w}\right)\mathbf{w}_k \right\|^2_2\\
  & + \beta_k \left( \nabla F(\mathbf{w}_k) + \frac{1}{\gamma_k} \left(\mathcal{L}\otimes I_{d_w}\right)\mathbf{w}_k \right)^\top\left(\mathcal{L}\otimes I_{d_w}\right)\mathbf{e}_k \qquad \textnormal{a.s.}
\end{split}\label{Eq:dV1a}
\end{align}
Based on Assumptions~\ref{Assump:Grad1} and \ref{Assump:Grad2}, there exists scalars ${\mu}_{v_1} \geq 0$ and ${\mu}_{v_2} \geq 0$ such that
\begin{align}
\begin{split}
&\mathbb{E}_{\xi} \left[ \| \mathbf{g}(\mathbf{w}_k,\bm{\xi}_k) + \frac{1}{\gamma_k} \left(\mathcal{L}\otimes I_{d_w}\right)\mathbf{w}_k \|_2^2 \right] \leq
\mu_{v_1} + \mu_{v_2} \left\|  \nabla F(\mathbf{w}_k) + \frac{1}{\gamma_k} \left(\mathcal{L}\otimes I_{d_w}\right)\mathbf{w}_k \right\|^2_2 \quad \textnormal{a.s.}
\end{split}\label{Eq:2ndGrad1}
\end{align}
Thus from \eqref{Eq:dV1a} we have
\begin{align}
\begin{split}
  &\mathbb{E}_{\xi} \left[ V(\gamma_k,\mathbf{w}_{k+1}) \right] - V(\gamma_k,\mathbf{w}_k) \leq \left( c_k \mu_{v_2} -  \mu \right)\alpha_k
  \left\|  \nabla F(\mathbf{w}_k) + \frac{1}{\gamma_k} \left(\mathcal{L}\otimes I_{d_w}\right)\mathbf{w}_k \right\|^2_2 + c_k\alpha_k\mu_{v_1} \\
  &+ \alpha_k c_k \left\| \frac{1}{\gamma_k} \left(\mathcal{L}\otimes I_{d_w}\right)\mathbf{e}_k \right\|_2^2
  + \beta_k \left( \nabla F(\mathbf{w}_k) + \frac{1}{\gamma_k} \left(\mathcal{L}\otimes I_{d_w}\right)\mathbf{w}_k \right)^\top \left(\mathcal{L}\otimes I_{d_w}\right)\mathbf{e}_k \quad \textnormal{a.s.}
\end{split}\label{Eq:dV2}
\end{align}
Note that
\begin{align}
    \alpha_k c_k \left\| \frac{1}{\gamma_k} \left(\mathcal{L}\otimes I_{d_w}\right)\mathbf{e}_k \right\|_2^2 &\leq \frac{\alpha_k c_k}{\gamma_k^2} \sigma_{\max}(\mathcal{L}) \left\| \mathbf{e}_k \right\|_2^2 \leq n^2\upsilon_0^2 \alpha_k \beta^2_k c_k \sigma_{\max}(\mathcal{L}),
\end{align}
and
\begin{align}
    \beta_k \nabla F^\top(\mathbf{w}_k) \left(\mathcal{L}\otimes I_{d_w}\right)\mathbf{e}_k &\leq \beta_k \sigma_{\max}(\mathcal{L}) \| \nabla F(\mathbf{w}_k) \|_2 \left\| \mathbf{e}_k \right\|_2
    \leq \alpha_k \beta_k n \upsilon_0 \sigma_{\max}(\mathcal{L}) \| \nabla F(\mathbf{w}_k) \|_2.
\end{align}
Thus based on Assumption \ref{Assump:AlphaBeta} we have $\alpha_k \beta_k$ is a summable sequence and therefore the above two terms are summable. Also note that
\begin{align}
    \frac{\beta_k}{\gamma_k} \mathbf{w}_k^\top \left(\mathcal{L}\otimes I_{d_w}\right) \left(\mathcal{L}\otimes I_{d_w}\right)\mathbf{e}_k &\leq \frac{\beta_k^2 }{\alpha_k} \sigma_{\max}(\mathcal{L})^2 \| \tilde{\mathbf{w}}_k \|_2 \| \mathbf{e}_k \|_2
\end{align}
From \eqref{Eq:MSbound1} we have
\begin{align}\label{Eq:MSbound1b}
    \mathbb{E} \left[ \| \tilde{\mathbf{w}}_k \|_2 \right]^2 \leq \mathbb{E} \left[ \| \tilde{\mathbf{w}}_k \|_2^2 \right] \leq \mu_w \frac{1}{(k+1)^{\delta_0}}, \quad \forall\,0 \leq \delta_0 < 2\delta_2-3\delta_1.
\end{align}
Thus we have
\begin{align}
    \frac{\beta_k}{\gamma_k} \mathbb{E} \left[ \mathbf{w}_k^\top \left(\mathcal{L}\otimes I_{d_w}\right) \left(\mathcal{L}\otimes I_{d_w}\right)\mathbf{e}_k \right] &\leq
    \frac{\beta_k^2 \sigma_{\max}(\mathcal{L})^2 }{\alpha_k}   \frac{\sqrt{\mu_w}}{(k+1)^{\delta_0/2}}  \frac{a\upsilon_0 n}{(k+1)^{\delta_2}}\\
    &=      \frac{\upsilon_0 n \sqrt{\mu_w}\sigma_{\max}(\mathcal{L})^2 b^2}{(k+1)^{\delta_2 + \delta_0/2 - \delta_2 + 2\delta_1}}
\end{align}
Thus from Assumption \ref{Assump:AlphaBeta} we have $\alpha_k \sqrt{\beta_k}$ is a summable and therefore $\frac{\beta_k}{\gamma_k} \mathbb{E} \left[ \mathbf{w}_k^\top \left(\mathcal{L}\otimes I_{d_w}\right) \left(\mathcal{L}\otimes I_{d_w}\right)\mathbf{e}_k \right]$ is also summable.

Substituting $\nabla V(\gamma_k, \mathbf{w}_k) = \nabla F(\mathbf{w}_k) + \displaystyle\frac{1}{\gamma_k} \left(\mathcal{L}\otimes I_{d_w}\right)\mathbf{w}_k$ and taking the total expectation of \eqref{Eq:dV2} yields
\begin{align}
\begin{split}
  &\mathbb{E}\left[ V(\gamma_k,\mathbf{w}_{k+1}) \right] - \mathbb{E}\left[ V(\gamma_k,\mathbf{w}_k)\right] \leq - \left( \mu - c_k \mu_{v_2} \right)\alpha_k  \mathbb{E}\left[ \left\|  \nabla V(\gamma_k, \mathbf{w}_k)  \right\|^2_2 \right] + \eta_k, 
\end{split}\label{Eq:dV3}
\end{align}
where $\eta_k$ denotes the reminding terms in \eqref{Eq:dV2} and we have already shown that $\eta_k$ is a summable sequence.

Note that
\begin{align}
\begin{split}
  &V(\gamma_{k+1},\mathbf{w}_{k+1}) = V(\gamma_k,\mathbf{w}_{k+1})  + \frac{1}{2} \left( \gamma_{k+1}^{-1} - \gamma_{k}^{-1} \right) \mathbf{w}^\top_{k+1} \left(\mathcal{L}\otimes I_{d_w}\right)\mathbf{w}_{k+1}
\end{split}\label{Eq:deltaV}
\end{align}
Combining \eqref{Eq:dV3} and \eqref{Eq:deltaV} yields
\begin{align}
\begin{split}
  \mathbb{E}\left[ V(\gamma_{k+1},\mathbf{w}_{k+1}) \right] - \mathbb{E}\left[ V(\gamma_k,\mathbf{w}_k)\right] \leq
  &- \left( \mu - c_k \mu_{v_2} \right)\alpha_k  \mathbb{E}\left[ \left\|  \nabla V(\gamma_k, \mathbf{w}_k)  \right\|^2_2 \right]  + \eta_k \\
  & + \frac{1}{2} \left( \gamma_{k+1}^{-1} - \gamma_{k}^{-1} \right) \mathbb{E}\left[ \mathbf{w}^\top_{k+1} \left(\mathcal{L}\otimes I_{d_w}\right)\mathbf{w}_{k+1} \right]
\end{split}\label{Eq:dV4}
\end{align}
\noindent If we select $\epsilon = \delta_2-\delta_1$, it follows directly from Lemma~\ref{Lemma10} that
  \begin{equation}
      \gamma_{k+1}^{-1} - \gamma_{k}^{-1} \leq \frac{2b\left(\delta_2-\delta_1\right)}{a} (k+1)^{\delta_2-\delta_1-1}.
  \end{equation}
Note that from Lemma~\ref{Lemma2} we have
\begin{align}    \mathbf{w}^\top_{k+1} \left(\mathcal{L}\otimes I_{d_w}\right)\mathbf{w}_{k+1} =
    \tilde{\mathbf{w}}^\top_{k+1} \left(\mathcal{L}\otimes I_{d_w}\right)\tilde{\mathbf{w}}_{k+1}
    \leq \sigma_{\max}\left( \mathcal{L} \right) \| \tilde{\mathbf{w}}_{k+1} \|_2^2.\end{align} Thus
\begin{align}
    &\frac{1}{2} \left( \gamma_{k+1}^{-1} - \gamma_{k}^{-1} \right) \mathbb{E}\left[ \mathbf{w}^\top_{k+1} \left(\mathcal{L}\otimes I_{d_w}\right)\mathbf{w}_{k+1} \right]
    \leq  \frac{2b\left(\delta_2-\delta_1\right)}{a}  (k+1)^{\delta_2-\delta_1-1} \sigma_{\max}\left( \mathcal{L} \right)  \mathbb{E}\left[ \| \tilde{\mathbf{w}}_{k+1} \|_2^2 \right]
\end{align}
We have established in \eqref{Eq:MSbound1} that for all $k\geq 0$
\begin{align}
    \mathbb{E} \left[ \| \tilde{\mathbf{w}}_k \|_2^2 \right] \leq \mu_w \frac{1}{(k+1)^{\delta_0}}, \quad \forall\,0 \leq \delta_0 < 2\delta_2-3\delta_1.
\end{align}
Therefore we have
\begin{align}
    &\frac{1}{2} \left( \gamma_{k+1}^{-1} - \gamma_{k}^{-1} \right) \mathbb{E}\left[ \mathbf{w}^\top_{k+1} \left(\mathcal{L}\otimes I_{d_w}\right)\mathbf{w}_{k+1} \right]
    \leq \frac{2b\left(\delta_2-\delta_1\right)}{a} (k+1)^{\delta_2-\delta_1-1} \sigma_{\max}\left( \mathcal{L} \right)  \mu_w \frac{1}{(k+1)^{\delta_0}}
\end{align}
Let $\mu_c = \frac{2b\left(\delta_2-\delta_1\right)}{a} \sigma_{\max}\left( \mathcal{L} \right)  \mu_w$. Now selecting $\delta_0 = 2\delta_2-3\delta_1-\varepsilon$, where $0 < \varepsilon \ll \delta_1 $, yields
\begin{align}
    \frac{1}{2} \left( \gamma_{k+1}^{-1} - \gamma_{k}^{-1} \right) \mathbb{E}\left[ \mathbf{w}^\top_{k+1} \left(\mathcal{L}\otimes I_{d_w}\right)\mathbf{w}_{k+1} \right]
    &\leq
    \mu_c (k+1)^{\delta_2-\delta_1-1-2\delta_2+3\delta_1+\varepsilon} \\&= \mu_c (k+1)^{-\delta_2+2\delta_1-1+\varepsilon}
\end{align}
Thus if we select $\delta_1$ and $\delta_2$ such that $\delta_2 > 2\delta_1 + \varepsilon$, then we have
\begin{equation}
    \frac{1}{2} \left( \gamma_{k+1}^{-1} - \gamma_{k}^{-1} \right) \mathbb{E}\left[ \mathbf{w}^\top_{k+1} \left(\mathcal{L}\otimes I_{d_w}\right)\mathbf{w}_{k+1} \right]
    \leq  \mu_c \frac{1}{(k+1)^{1+\varepsilon_1}},
\end{equation}
where $\varepsilon_1 > 0$ and $\delta_2 - 2\delta_1 - \varepsilon = \varepsilon_1$. Now we can write \eqref{Eq:dV4} as
\begin{align}
\begin{split}
  &\mathbb{E}\left[ V(\gamma_{k+1},\mathbf{w}_{k+1}) \right] - \mathbb{E}\left[ V(\gamma_k,\mathbf{w}_k)\right] \leq - \left( \mu - c_k \mu_{v_2} \right)
  \alpha_k \mathbb{E}\left[ \left\|  \nabla V(\gamma_k, \mathbf{w}_k)  \right\|^2_2 \right] + \eta_k +  \frac{\mu_c}{(k+1)^{1+\varepsilon_1}}
\end{split}\label{Eq:dV5}
\end{align}
Since $c_k$ is decreasing to zero, for sufficiently large $k$, we have $c_k \mu_{v_2} < \mu/2$. Therefore $\left( \mu - c_k \mu_{v_2} \right) > \frac{1}{2} \mu$ for sufficiently large $k$. Thus we have
\begin{align}
\begin{split}
  &\mathbb{E}\left[ V(\gamma_{k+1},\mathbf{w}_{k+1}) \right] - \mathbb{E}\left[ V(\gamma_k,\mathbf{w}_k)\right] \leq - \frac{1}{2} \mu\alpha_k \mathbb{E}\left[ \left\|  \nabla V(\gamma_k, \mathbf{w}_k)  \right\|^2_2 \right] + \eta_k   +  \frac{\mu_c}{(k+1)^{1+\varepsilon_1}}
\end{split}\label{Eq:dV6}
\end{align}
Now \eqref{Eq:dV6} can be written in the form of \eqref{Eq:Robbins} after selecting $a_k = 0$,
\begin{align}
    w_k &= \eta_k   +  \frac{\mu_c}{(k+1)^{1+\varepsilon_1}},\\
    u_k &= \frac{1}{2} \mu\alpha_k \mathbb{E}\left[ \left\|  \nabla V(\gamma_k, \mathbf{w}_k)  \right\|^2_2 \right].
\end{align}
Note that here we have $a_k = 0$, $u_k\geq 0$ and $w_k\geq 0$ with $\sum\limits_{k=0}^{\infty}a_k < \infty $ and $\sum\limits_{k=0}^{\infty}w_k < \infty $. Therefore from Lemma~\ref{Lemma:Robbins} we have $\mathbb{E}\left[ V(\gamma_k,\mathbf{w}_k)\right]$ is a convergent sequence and  $$\sum\limits_{k=0}^{\infty} \,\frac{1}{2} \mu \alpha_k \mathbb{E}\left[ \left\|  \nabla V(\gamma_k,\mathbf{w}_k)  \right\|^2_2 \right] < \infty .$$

\section{Proof of Theorem \ref{Theorem:Convergence}}\label{Appdx:D}
Note that
    \begin{align}
      \left\| \mathbf{w}_{k + 1} - \mathbf{w}_{k} \right\|_2^2 &= \alpha_k^2  \left\| \mathbf{g}(\mathbf{w}_k,\bm{\xi}_k) + \frac{1}{\gamma_k} \left(\mathcal{L}\otimes I_{d_w}\right)\mathbf{w}_k
      - \frac{1}{\gamma_k} \left(\mathcal{L}\otimes I_{d_w}\right)\mathbf{e}_k \right\|_2^2  \\
      &\leq 2 \alpha_k^2  \left\| \mathbf{g}(\mathbf{w}_k,\bm{\xi}_k) + \frac{1}{\gamma_k} \left(\mathcal{L}\otimes I_{d_w}\right)\mathbf{w}_k \right\|_2^2 +
      2 \beta_k^2  \left\|  \left(\mathcal{L}\otimes I_{d_w}\right)\mathbf{e}_k \right\|_2^2 \\
      &\leq 2 \alpha_k^2  \left\| \mathbf{g}(\mathbf{w}_k,\bm{\xi}_k) + \frac{1}{\gamma_k} \left(\mathcal{L}\otimes I_{d_w}\right)\mathbf{w}_k \right\|_2^2 +
      2 \left(n\upsilon_0 \alpha_k \beta_k  \sigma_{\max}\left(\mathcal{L}\right) \right)^2\label{Eq:DSG3}
    \end{align}
    Now form \eqref{Eq:2ndGrad1}, using the tower rule yields
    \begin{align}
    \begin{split}
    &\mathbb{E} \left[ \| \mathbf{g}(\mathbf{w}_k,\bm{\xi}_k) + \frac{1}{\gamma_k} \left(\mathcal{L}\otimes I_{d_w}\right)\mathbf{w}_k \|_2^2 \right] \leq
    \mu_{v_1} + \mu_{v_2} \mathbb{E} \left[\left\|  \nabla F(\mathbf{w}_k) + \frac{1}{\gamma_k} \left(\mathcal{L}\otimes I_{d_w}\right)\mathbf{w}_k \right\|^2_2 \right]
    \end{split}\label{Eq:2ndGrad2}
    \end{align}
    Now taking the expectation of \eqref{Eq:DSG3} and substituting \eqref{Eq:2ndGrad2} yields
    \begin{align}\label{InEq:DSG1}
      \mathbb{E} \left[ \left\| \mathbf{w}_{k + 1} - \mathbf{w}_{k} \right\|_2^2 \right] &\leq 2\alpha_k^2 \mu_{v_1} +
      2 \left(n\upsilon_0 \alpha_k \beta_k  \sigma_{\max}\left(\mathcal{L}\right) \right)^2 + 2\alpha_k^2 \mu_{v_2} \mathbb{E} \left[\left\|  \nabla V(\mathbf{w}_k) \right\|^2_2 \right].
    \end{align}
    Thus we have
    \begin{align}\label{InEq:DSG2}
    \begin{split}
      \sum\limits_{k=0}^{\infty} \, \mathbb{E} \left[ \left\| \mathbf{w}_{k + 1} - \mathbf{w}_{k} \right\|_2^2 \right] &\leq 2\sum\limits_{k=0}^{\infty} \,\left(\alpha_k^2 \mu_{v_1} + \left(n\upsilon_0 \alpha_k \beta_k  \sigma_{\max}\left(\mathcal{L}\right) \right)^2\right) \\
      &\qquad \qquad + 2\sum\limits_{k=0}^{\infty} \, \left(\alpha_k^2 \mu_{v_2} \mathbb{E} \left[\left\|  \nabla V(\mathbf{w}_k) \right\|^2_2 \right]\right).
    \end{split}
    \end{align}
    Now \eqref{InEq:DSG3} follows from \eqref{Eqn:SummableGrad} and from noting that $\alpha_k$ is square summable. Furthermore, since every summable sequence is convergent, we have \eqref{ConvergenceEqn}.


\section{Proof of Theorem~\ref{Theorem:SummableGrad}}\label{Appdx:E}
Taking the conditional expectation $\mathbb{E}_{\xi}[\cdot]$ of \eqref{Eq:DSG2}  yields
\begin{align}\label{Eq:DSG2b}
  \mathbb{E}_{\xi}\left[\, \mathbf{w}_{k + 1}  - \mathbf{w}_{k} \,\right]
  &=  - \alpha_k \nabla V(\gamma_k,\mathbf{w}_k) + \beta_k \left(\mathcal{L}\otimes I_{d_w}\right)\mathbf{e}_k \quad \textnormal{a.s.}
\end{align}
Thus we have
\begin{align}
  \left\| \mathbb{E}_{\xi}\left[\, \mathbf{w}_{k + 1}  - \mathbf{w}_{k} \,\right] \right\|_2
  &\leq  \alpha_k  \left\| \nabla V(\gamma_k,\mathbf{w}_k) \right\|_2 + \beta_k \left\| \left(\mathcal{L}\otimes I_{d_w}\right)\mathbf{e}_k \right\|_2 \leq  \alpha_k  \left\| \nabla V(\gamma_k,\mathbf{w}_k) \right\|_2 + \beta_k \sigma_{\max}\left(\mathcal{L}\right) \left\| \mathbf{e}_k \right\|_2 \quad \textnormal{a.s.}
\end{align}
From the triggering condition \eqref{TriggCond} we have $\left\|\mathbf{e}(k)\right\|_1 < n\upsilon_0\alpha_k$. Thus we have
\begin{align}\label{Eq:DSG23}
  \left\| \mathbb{E}_{\xi}\left[\, \mathbf{w}_{k + 1}  - \mathbf{w}_{k} \,\right] \right\|^2_2
  &\leq  2\alpha_k^2  \left(\left\| \nabla V(\gamma_k,\mathbf{w}_k) \right\|^2_2 + n^2\upsilon_0^2 \sigma_{\max}\left(\mathcal{L}\right)^2 \beta_k^2 \right)\quad \textnormal{a.s.}
\end{align}
Therefore
\begin{align}\label{Eq:DSG2d}
  \mathbb{E}\left[ \left\|  \nabla V(\gamma_k,\mathbf{w}_k)  \right\|^2_2 \right] \geq \frac{1}{2} \alpha_k^{-2} \, \mathbb{E}\left[ \left\| \mathbb{E}_{\xi}\left[\, \mathbf{w}_{k + 1}  - \mathbf{w}_{k} \,\right] \right\|^2_2 \right] - n^2\upsilon_0^2 \sigma_{\max}\left(\mathcal{L}\right)^2 \beta_k^2
\end{align}
Multiplying by $\alpha_k$ and taking the summation yields
\begin{align}
  \sum\limits_{k=0}^{\infty} \, \alpha_k \mathbb{E}\left[ \left\|  \nabla V(\gamma_k,\mathbf{w}_k)  \right\|^2_2 \right] \geq \frac{1}{2} \sum\limits_{k=0}^{\infty} \,\alpha_k^{-1} \, \mathbb{E}\left[ \left\| \mathbb{E}_{\xi}\left[\, \mathbf{w}_{k + 1}  - \mathbf{w}_{k} \,\right] \right\|^2_2 \right] - n^2\upsilon_0^2 \sigma_{\max}\left(\mathcal{L}\right)^2 \sum\limits_{k=0}^{\infty} \,\alpha_k\beta_k^2
\end{align}
Since $\alpha_k\beta_k^2$ is summable, it follows from \eqref{Eqn:SummableGrad} that
\begin{equation}\label{Eqn:Summabledw}
    \sum\limits_{k=0}^{\infty} \, \alpha_k^{-1}\, \mathbb{E}\left[ \left\| \mathbb{E}_{\xi}\left[\, \mathbf{w}_{k + 1}  - \mathbf{w}_{k} \,\right] \right\|^2_2 \right] < \infty.
\end{equation}
Now note that $\bar{\mathbf{w}}_{k + 1} - \bar{\mathbf{w}}_{k}  =  \frac{1}{n} \left(\mathbf{1}_n \mathbf{1}_n^\top \otimes I_{d_w} \right) \left( \mathbf{w}_{k + 1} - \mathbf{w}_{k} \right)$. Thus
$\mathbb{E}_{\xi}\left[\ \bar{\mathbf{w}}_{k + 1} - \bar{\mathbf{w}}_{k} \right]  =  \frac{1}{n} \left(\mathbf{1}_n \mathbf{1}_n^\top \otimes I_{d_w} \right) \mathbb{E}_{\xi}\left[\ \mathbf{w}_{k + 1} - \mathbf{w}_{k}\right]$ a.s. and $\left\| \mathbb{E}_{\xi}\left[\ \bar{\mathbf{w}}_{k + 1} - \bar{\mathbf{w}}_{k} \right] \right\|_2 \leq  \left\| \mathbb{E}_{\xi}\left[\ \mathbf{w}_{k + 1} - \mathbf{w}_{k}\right] \right\|_2$ a.s.
Therefore it follows from \eqref{Eqn:Summabledw} that
\begin{equation}\label{Eqn:Summabledwbar}
    \sum\limits_{k=0}^{\infty} \, \alpha_k^{-1}\, \mathbb{E}\left[ \left\| \mathbb{E}_{\xi}\left[\, \bar{\mathbf{w}}_{k + 1}  - \bar{\mathbf{w}}_{k} \,\right] \right\|^2_2 \right] < \infty.
\end{equation}
From \eqref{Eq:DSG2b} we have
\begin{align}\label{Eq:DSGbar}
  \mathbb{E}_{\xi}\left[\, \bar{\mathbf{w}}_{k + 1}  - \bar{\mathbf{w}}_{k} \,\right] = \frac{1}{n} \left(\mathbf{1}_n \mathbf{1}_n^\top \otimes I_{d_w} \right)  \mathbb{E}_{\xi}\left[\, \mathbf{w}_{k + 1} - \mathbf{w}_{k} \,\right]
  &= - \alpha_k \overline{\nabla F} (\mathbf{w}_k) \quad \textnormal{a.s.}
\end{align}
Now substituting \eqref{Eq:DSGbar} into \eqref{Eqn:Summabledwbar} yields \eqref{Eqn:SummableDFbar}.

\section{Proof of Theorem \ref{Theorem:ConvergenceRate}}\label{Appdx:D}

\noindent From Theorem \ref{Theorem:InfSum} we have
\begin{align}
    &\sum_{k=0}^{K}\,\alpha_k \mathbb{E}\left[ \left\| \sum_{i=1}^n \nabla f_i(\bm{w}_i(k))  \right\|_2^2 \right] \leq C ,
\end{align}
for some $K > 0$ and some positive constant $C < \infty$. Now dividing both sides of this inequality by $\sum\limits_{k=0}^{K} \, \alpha_k\,$ yields
\begin{align}
    \frac{1}{\sum_{k=0}^{K} \, \alpha_k\,}\sum_{k=0}^{K} \, \alpha_k\, \mathbb{E}\left[ \left\| \sum_{i=1}^n \nabla f_i(\bm{w}_i(k))  \right\|_2^2 \right] \leq \frac{C}{\sum_{k=0}^{K} \, \alpha_k\,}.
\end{align}
Notice that
\begin{align}
    \sum\limits_{k=0}^{K} \, \alpha_k\, &= \sum\limits_{k=0}^{K} \, \frac{a}{(k+1)^{\delta_2}}\,\geq \int_0^K\,\frac{a}{(x+1)^{\delta_2}}\,dx.
\end{align}
Note $\int_0^K\,\frac{a}{(x+1)^{\delta_2}}\,dx = a\log(K+1)$ for $\delta_2 = 1$ and $\int_0^K\,\frac{a}{(x+1)^{\delta_2}}\,dx= \frac{a}{1-\delta_2}((K+1)^{1-\delta_2}-1)$ if $\delta_2 \in (0.5, \,\, 1)$. Thus when $\delta_2 = 1$, we have
\begin{align}
    \sum_{k=0}^{K} \,  \frac{\alpha_k}{\sum_{j=0}^{K} \, \alpha_j\,} \, \mathbb{E}\left[ \left\| \sum_{i=1}^n \nabla f_i(\bm{w}_i(k))  \right\|_2^2 \right] \leq \frac{\hat{C}}{\log(K+1)},
\end{align}
where $\hat{C} < \infty$ is a positive constant. We therefore can show a weak convergence result, i.e.,
\begin{align}
    \min_{k\in\{0,1,\ldots,K\}} \,\, \mathbb{E}\left[ \left\| \sum_{i=1}^n \nabla f_i(\bm{w}_i(k))  \right\|_2^2 \right] \overset{K\rightarrow \infty}{\longrightarrow}0.
\end{align}
Sample a parameter $\mathbf{z}^{\tiny{K}}$ from $\{ \mathbf{w}_k\}_{k=0}^{K}$ for $k=0,1,\ldots,K$ with probability $\mathbb{P}\left( \mathbf{z}^{\tiny{K}} = \mathbf{w}_k\right) = \frac{\alpha_k}{\sum_{j=0}^{K} \, \alpha_j\,}$.  This gives
\begin{align}
    \mathbb{E}\left[ \left\| \sum_{i=1}^n \nabla f_i(\bm{z}_i^K)  \right\|_2^2 \right] = \sum_{k=0}^{K} \,  \frac{\alpha_k}{\sum_{j=0}^{K} \, \alpha_j\,} \, \mathbb{E}\left[ \left\| \sum_{i=1}^n \nabla f_i(\bm{w}_i(k))  \right\|_2^2 \right].
\end{align}
Therefore for $\delta_2=1$ we have
\begin{align}
    \mathbb{E}\left[ \left\| \sum_{i=1}^n \nabla f_i(\bm{z}_i^K)  \right\|_2^2 \right]  = O \left(\frac{1}{\log(K+1)}\right),
\end{align}
and for $\delta_2 \in (0.5, \,\, 1)$ we have
\begin{align}
    \mathbb{E}\left[ \left\| \sum_{i=1}^n \nabla f_i(\bm{z}_i^K)  \right\|_2^2 \right]  =O \left(\frac{1}{(K+1)^{1-\delta_2}}\right).
\end{align}
This concludes the proof of Theorem \ref{Theorem:ConvergenceRate}.

\section{Proof of Theorem~\ref{Theorem:OptCond}}\label{Appdx:F}
Define $G(\mathbf{w}_k) \triangleq \left\| \overline{\nabla F} (\mathbf{w}_k) \right\|^2_2$. Thus we have
\begin{align}
    \nabla G(\mathbf{w}) &= 2\nabla^2 F(\mathbf{w}) \mathcal{J}\, \nabla F(\mathbf{w}),\label{GradG}
\end{align}
where $\mathcal{J}\, = \left( \frac{1}{n} \left(\mathbf{1}_n \mathbf{1}_n^\top \otimes I_{d_w} \right) \right)$ and $\mathcal{J}^2 = \mathcal{J}\,$. Since $F(\cdot)$ is twice continuously differentiable and $\nabla F(\cdot)$ is Liptschitz continuous with constant $L$, we have $ \nabla^2 F(\mathbf{w}) \leq L I_{nd_w}$. Therefore $\forall \,\mathbf{w}_a,\,\mathbf{w}_b\in\mathbb{R}^{n d_w}$,
\begin{align}
    \nabla G(\mathbf{w}_{a}) - \nabla G(\mathbf{w}_b) & = 2\nabla^2 F(\mathbf{w}_{a}) \mathcal{J}\, \nabla F(\mathbf{w}_{a}) - 2\nabla^2 F(\mathbf{w}_b) \mathcal{J}\nabla F(\mathbf{w}_b) +     2\nabla^2 F(\mathbf{w}_{a}) \mathcal{J}\, \nabla F(\mathbf{w}_{b})   - 2\nabla^2 F(\mathbf{w}_{a}) \mathcal{J}\, \nabla F(\mathbf{w}_{b}) \\
    & = 2\nabla^2 F(\mathbf{w}_{a}) \mathcal{J}\, \left(\nabla F(\mathbf{w}_{a}) - \nabla F(\mathbf{w}_{b}) \right) + 2\left( \nabla^2 F(\mathbf{w}_{a}) - \nabla^2 F(\mathbf{w}_{b}) \right) \mathcal{J}\, \nabla F(\mathbf{w}_{b})
\end{align}
Since $\nabla^2 F(\mathbf{w}_{a})$ is Lipschitz continuous with constant $L_H$, and $\nabla F(\mathbf{w}_{b}) \leq \mu_F$, we have
\begin{align}
    &\left\| \nabla G(\mathbf{w}_{a}) - \nabla G(\mathbf{w}_b) \right\|_2 \leq 2L^2 \left\| \mathbf{w}_{a} - \mathbf{w}_{b} \right\|_2   + 2\mu_F L_H \left\| \mathbf{w}_{a} - \mathbf{w}_{b} \right\|_2\leq L_G \left\| \mathbf{w}_{a} - \mathbf{w}_{b} \right\|_2,
\end{align}
where $L_G \geq 2L^2 + 2\mu_F L_H$. Thus $\nabla G(\mathbf{w})$ is Lipschitz continuous and from Lemma~\ref{Lemma:Lipz} we have
\begin{align}
    G(\mathbf{w}_{k+1}) \leq  G(\mathbf{w}_k) &+ \nabla G(\mathbf{w}_k)^\top \left( \mathbf{w}_{k+1} - \mathbf{w}_{k} \right) + \frac{1}{2}L_G \left\| \mathbf{w}_{k+1} - \mathbf{w}_{k} \right\|_2^2
\end{align}
Now substituting \eqref{GradG} and taking the conditional expectation $\mathbb{E}_{\xi}[\,\cdot\,]$ yields
\begin{align}
    \mathbb{E}_{\xi}\left[\,G(\mathbf{w}_{k+1}) \,\right] \leq  &G(\mathbf{w}_k) + 2\nabla F(\mathbf{w})^\top \mathcal{J}\, \nabla^2 F(\mathbf{w}) \mathbb{E}_{\xi}\left[\, \mathbf{w}_{k+1} - \mathbf{w}_{k} \,\right] + \frac{1}{2}L_G \mathbb{E}_{\xi}\left[\,\left\| \mathbf{w}_{k+1} - \mathbf{w}_{k} \right\|_2^2 \,\right]
\end{align}
Since $\nabla F(\mathbf{w})^\top \mathcal{J}\, = \nabla V(\gamma_k,\mathbf{w})^\top \mathcal{J}$, substituting \eqref{Eq:DSG2b} yields
\begin{align}
\begin{split}
   &\mathbb{E}_{\xi}\left[\,G(\mathbf{w}_{k+1}) \,\right]  \leq  G(\mathbf{w}_k)  + \frac{1}{2}L_G \mathbb{E}_{\xi}\left[\,\left\| \mathbf{w}_{k+1} - \mathbf{w}_{k} \right\|_2^2 \,\right] - 2\alpha_k \,\nabla V(\gamma_k,\mathbf{w})^\top \mathcal{J}\, \nabla^2 F(\mathbf{w}) \nabla V(\gamma_k,\mathbf{w})\\
    &\qquad\qquad\qquad\qquad\qquad\qquad\qquad\qquad \qquad\qquad + 2\beta_k \,\nabla V(\gamma_k,\mathbf{w})^\top \mathcal{J}\, \nabla^2 F(\mathbf{w}) \left(\mathcal{L}\otimes I_{d_w}\right)\mathbf{e}_k
    \end{split}\\
    &\leq G(\mathbf{w}_k)  + \frac{1}{2}L_G \mathbb{E}_{\xi}\left[\,\left\| \mathbf{w}_{k+1} - \mathbf{w}_{k} \right\|_2^2 \,\right] + 2\alpha_k L \left\| \nabla V(\gamma_k,\mathbf{w}) \right\|^2_2
    + \alpha_k\beta_k  \mu_{\upsilon}  \left\| \nabla V(\gamma_k,\mathbf{w}) \right\|_2 ,
\end{align}
where $\mu_{\upsilon} = 2 n\upsilon_0 L \sigma_{\max}(\mathcal{L}) $. Now taking the total expectation yields
\begin{align}\label{Eqn:DiffG}
\begin{split}
    \mathbb{E}\left[\,G(\mathbf{w}_{k+1}) \,\right]
    \leq \mathbb{E}\left[\, G(\mathbf{w}_k) \,\right] + \frac{1}{2}L_G \mathbb{E}\left[\,\left\| \mathbf{w}_{k+1} - \mathbf{w}_{k} \right\|_2^2 \,\right] &+ 2\alpha_k L \mathbb{E}\left[\,\left\| \nabla V(\gamma_k,\mathbf{w}) \right\|^2_2 \,\right]\\ &+ \alpha_k\beta_k  \mu_{\upsilon}  \mathbb{E}\left[\,\left\| \nabla V(\gamma_k,\mathbf{w}) \right\|_2 \,\right]
\end{split}
\end{align}
From \eqref{Eqn:SummableGrad} and \eqref{InEq:DSG3}, we know that $\alpha_k \displaystyle \mathbb{E}\left[\,\left\| \nabla V(\gamma_k,\mathbf{w}) \right\|^2_2 \,\right]$ and $\displaystyle \mathbb{E}\left[\,\left\| \mathbf{w}_{k+1} - \mathbf{w}_{k} \right\|_2^2 \,\right]$ are summable. Since $\alpha_k\sqrt{\beta_k}$ is summable and $\sqrt{\beta_k}\displaystyle\mathbb{E}\left[\,\left\| \nabla V(\gamma_k,\mathbf{w}) \right\|_2 \,\right]$ is bounded, \eqref{Eqn:DiffG} can be written in the form of \eqref{Eq:Robbins} and it follows from Lemma~\ref{Lemma:Robbins} that $\displaystyle \mathbb{E}\left[\, G(\mathbf{w}_k) \,\right]$ converges. Since $ \displaystyle \mathbb{E}\left[\, G(\mathbf{w}_k) \,\right] = \displaystyle \mathbb{E}\left[\, \left\| \overline{\nabla F} (\mathbf{w}_k) \right\|^2_2 \,\right]$ it follows from Theorem~\ref{Theorem:SummableGrad} that $\displaystyle\mathbb{E}\left[\, G(\mathbf{w}_k) \,\right]$ must converge to zero.

\end{document}